\documentclass[reqno,12pt]{amsart}

\numberwithin{equation}{section}

\newtheorem{Theorem}{Theorem}

\newtheorem{Lemma}[Theorem]{Lemma}
\newtheorem{Corollary}[Theorem]{Corollary}

\theoremstyle{definition}

\theoremstyle{remark}

\newtheorem{Remark}[Theorem]{Remark}
\newtheorem{Remarks}[Theorem]{Remarks}

\newtheorem{Example}[Theorem]{Example}

\def\P{{\mathcal P}}
\def\Z{\mathbb{Z}}
\def\GF{\operatorname{GF}}
\def\NE{\operatorname{NE}}

\catcode`\@=11

\thinlines
\newskip\Einheit \Einheit=0.6cm
\newcount\xcoord \newcount\ycoord
\newdimen\xdim \newdimen\ydim \newdimen\PfadD@cke \newdimen\Pfadd@cke
\def\PfadDicke#1{\PfadD@cke#1 \divide\PfadD@cke by2 \Pfadd@cke\PfadD@cke \multiply\PfadD@cke by2}
\long\def\LOOP#1\REPEAT{\def\BODY{#1}\ITERATE}
\def\ITERATE{\BODY \let\next\ITERATE \else\let\next\relax\fi \next}
\let\REPEAT=\fi
\def\Punkt{\hbox{\raise-2pt\hbox to0pt{\hss\scriptsize$\bullet$\hss}}}
\def\FeinPunkt(#1,#2){\unskip
      \raise#2\Einheit\hbox to0pt{\hskip#1\Einheit\Punkt\hss}\hss}
\def\DuennPunkt(#1,#2){\unskip
  \raise#2 \Einheit\hbox to0pt{\hskip#1 \Einheit
          \raise-2.5pt\hbox to0pt{\hss\normalsize$\bullet$\hss}\hss}}
\def\NormalPunkt(#1,#2){\unskip
  \raise#2 \Einheit\hbox to0pt{\hskip#1 \Einheit
          \raise-3pt\hbox to0pt{\hss\large$\bullet$\hss}\hss}}
\def\DickPunkt(#1,#2){\unskip
  \raise#2 \Einheit\hbox to0pt{\hskip#1 \Einheit
          \raise-4pt\hbox to0pt{\hss\Large$\bullet$\hss}\hss}}
\def\Kreis(#1,#2){\unskip
  \raise#2 \Einheit\hbox to0pt{\hskip#1 \Einheit
          \raise-4pt\hbox to0pt{\hss\Large$\circ$\hss}\hss}}
\def\Diagonale(#1,#2)#3{\unskip\leavevmode
  \xcoord#1\relax \ycoord#2\relax
      \raise\ycoord \Einheit\hbox to0pt{\hskip\xcoord \Einheit
         \unitlength\Einheit
         \line(1,1){#3}\hss}}
\def\AntiDiagonale(#1,#2)#3{\unskip\leavevmode
  \xcoord#1\relax \ycoord#2\relax \advance\xcoord by -0.05\relax
      \raise\ycoord \Einheit\hbox to0pt{\hskip\xcoord \Einheit
         \unitlength\Einheit
         \line(1,-1){#3}\hss}}
\def\Pfad(#1,#2),#3\endPfad{\unskip\leavevmode
  \xcoord#1 \ycoord#2 \thicklines\ZeichnePfad#3\endPfad\thinlines}
\def\ZeichnePfad#1{\ifx#1\endPfad\let\next\relax
  \else\let\next\ZeichnePfad
    \ifnum#1=1
      \raise\ycoord \Einheit\hbox to0pt{\hskip\xcoord \Einheit
         \vrule height\Pfadd@cke width1 \Einheit depth\Pfadd@cke\hss}%
      \advance\xcoord by 1
    \else\ifnum#1=2
      \raise\ycoord \Einheit\hbox to0pt{\hskip\xcoord \Einheit
        \hbox{\hskip-1pt\vrule height1 \Einheit width\PfadD@cke depth0pt}\hss}%
      \advance\ycoord by 1
    \else\ifnum#1=3
      \raise\ycoord \Einheit\hbox to0pt{\hskip\xcoord \Einheit
         \unitlength\Einheit
         \line(1,1){1}\hss}
      \advance\xcoord by 1
      \advance\ycoord by 1
    \else\ifnum#1=4
      \raise\ycoord \Einheit\hbox to0pt{\hskip\xcoord \Einheit
         \unitlength\Einheit
         \line(1,-1){1}\hss}
      \advance\xcoord by 1
      \advance\ycoord by -1
    \fi\fi\fi\fi
  \fi\next}
\def\hSSchritt{\leavevmode\raise-0.4pt\hbox to0pt{\hss.\hss}\hskip0.2\Einheit
  \raise-0.4pt\hbox to0pt{\hss.\hss}\hskip0.2\Einheit
  \raise-0.4pt\hbox to0pt{\hss.\hss}\hskip0.2\Einheit
  \raise-0.4pt\hbox to0pt{\hss.\hss}\hskip0.2\Einheit
  \raise-0.4pt\hbox to0pt{\hss.\hss}\hskip0.2\Einheit}
\def\vSSchritt{\vbox{\baselineskip0.2\Einheit\lineskiplimit0pt
\hbox{.}\hbox{.}\hbox{.}\hbox{.}\hbox{.}}}
\def\DSSchritt{\leavevmode\raise-0.4pt\hbox to0pt{%
  \hbox to0pt{\hss.\hss}\hskip0.2\Einheit
  \raise0.2\Einheit\hbox to0pt{\hss.\hss}\hskip0.2\Einheit
  \raise0.4\Einheit\hbox to0pt{\hss.\hss}\hskip0.2\Einheit
  \raise0.6\Einheit\hbox to0pt{\hss.\hss}\hskip0.2\Einheit
  \raise0.8\Einheit\hbox to0pt{\hss.\hss}\hss}}
\def\dSSchritt{\leavevmode\raise-0.4pt\hbox to0pt{%
  \hbox to0pt{\hss.\hss}\hskip0.2\Einheit
  \raise-0.2\Einheit\hbox to0pt{\hss.\hss}\hskip0.2\Einheit
  \raise-0.4\Einheit\hbox to0pt{\hss.\hss}\hskip0.2\Einheit
  \raise-0.6\Einheit\hbox to0pt{\hss.\hss}\hskip0.2\Einheit
  \raise-0.8\Einheit\hbox to0pt{\hss.\hss}\hss}}
\def\SPfad(#1,#2),#3\endSPfad{\unskip\leavevmode
  \xcoord#1 \ycoord#2 \ZeichneSPfad#3\endSPfad}
\def\ZeichneSPfad#1{\ifx#1\endSPfad\let\next\relax
  \else\let\next\ZeichneSPfad
    \ifnum#1=1
      \raise\ycoord \Einheit\hbox to0pt{\hskip\xcoord \Einheit
         \hSSchritt\hss}%
      \advance\xcoord by 1
    \else\ifnum#1=2
      \raise\ycoord \Einheit\hbox to0pt{\hskip\xcoord \Einheit
        \hbox{\hskip-2pt \vSSchritt}\hss}%
      \advance\ycoord by 1
    \else\ifnum#1=3
      \raise\ycoord \Einheit\hbox to0pt{\hskip\xcoord \Einheit
         \DSSchritt\hss}
      \advance\xcoord by 1
      \advance\ycoord by 1
    \else\ifnum#1=4
      \raise\ycoord \Einheit\hbox to0pt{\hskip\xcoord \Einheit
         \dSSchritt\hss}
      \advance\xcoord by 1
      \advance\ycoord by -1
    \fi\fi\fi\fi
  \fi\next}
\def\Koordinatenachsen(#1,#2){\unskip
 \hbox to0pt{\hskip-0.5pt\vrule height#2 \Einheit width0.5pt depth1 \Einheit}%
 \hbox to0pt{\hskip-1 \Einheit \xcoord#1 \advance\xcoord by1
    \vrule height0.25pt width\xcoord \Einheit depth0.25pt\hss}}
\def\Koordinatenachsen(#1,#2)(#3,#4){\unskip
 \hbox to0pt{\hskip-0.5pt \ycoord-#4 \advance\ycoord by1
    \vrule height#2 \Einheit width0.5pt depth\ycoord \Einheit}%
 \hbox to0pt{\hskip-1 \Einheit \hskip#3\Einheit 
    \xcoord#1 \advance\xcoord by1 \advance\xcoord by-#3 
    \vrule height0.25pt width\xcoord \Einheit depth0.25pt\hss}}
\def\Gitter(#1,#2){\unskip \xcoord0 \ycoord0 \leavevmode
  \LOOP\ifnum\ycoord<#2
    \loop\ifnum\xcoord<#1
      \raise\ycoord \Einheit\hbox to0pt{\hskip\xcoord \Einheit\Punkt\hss}%
      \advance\xcoord by1
    \repeat
    \xcoord0
    \advance\ycoord by1
  \REPEAT}
\def\Gitter(#1,#2)(#3,#4){\unskip \xcoord#3 \ycoord#4 \leavevmode
  \LOOP\ifnum\ycoord<#2
    \loop\ifnum\xcoord<#1
      \raise\ycoord \Einheit\hbox to0pt{\hskip\xcoord \Einheit\Punkt\hss}%
      \advance\xcoord by1
    \repeat
    \xcoord#3
    \advance\ycoord by1
  \REPEAT}
\def\Label#1#2(#3,#4){\unskip \xdim#3 \Einheit \ydim#4 \Einheit
  \def\lo{\advance\xdim by-0.5 \Einheit \advance\ydim by0.5 \Einheit}%
  \def\llo{\advance\xdim by-0.25cm \advance\ydim by0.5 \Einheit}%
  \def\loo{\advance\xdim by-0.5 \Einheit \advance\ydim by0.25cm}%
  \def\o{\advance\ydim by0.25cm}%
  \def\ro{\advance\xdim by0.5 \Einheit \advance\ydim by0.5 \Einheit}%
  \def\rro{\advance\xdim by0.25cm \advance\ydim by0.5 \Einheit}%
  \def\roo{\advance\xdim by0.5 \Einheit \advance\ydim by0.25cm}%
  \def\l{\advance\xdim by-0.30cm}%
  \def\r{\advance\xdim by0.30cm}%
  \def\lu{\advance\xdim by-0.5 \Einheit \advance\ydim by-0.6 \Einheit}%
  \def\llu{\advance\xdim by-0.25cm \advance\ydim by-0.6 \Einheit}%
  \def\luu{\advance\xdim by-0.5 \Einheit \advance\ydim by-0.30cm}%
  \def\u{\advance\ydim by-0.30cm}%
  \def\ru{\advance\xdim by0.5 \Einheit \advance\ydim by-0.6 \Einheit}%
  \def\rru{\advance\xdim by0.25cm \advance\ydim by-0.6 \Einheit}%
  \def\ruu{\advance\xdim by0.5 \Einheit \advance\ydim by-0.30cm}%
  #1\raise\ydim\hbox to0pt{\hskip\xdim
     \vbox to0pt{\vss\hbox to0pt{\hss$#2$\hss}\vss}\hss}%
}
\catcode`\@=12

\begin{document}

\newbox\Adr
\setbox\Adr\vbox{
\centerline{\sc Sudhir R. Ghorpade$^{\ast}$ and  Christian~Krattenthaler$^{\dagger}$}
\vskip18pt
\centerline{Department of Mathematics,}
\centerline{Indian Institute of Technology Bombay,}
\centerline{Powai, Mumbai 400076, India.}
\centerline{{\it URL: } \footnotesize{\tt http://www.math.iitb.ac.in/\~{}srg}/}
\vskip10pt
\centerline{Fakult\"at f\"ur Mathematik, Universit\"at Wien}
\centerline{Oskar-Morgenstern-Platz~1, A-1090 Vienna, Austria.}
\centerline{{\it URL: }  \footnotesize{\tt
    http://www.mat.univie.ac.at/\lower0.5ex\hbox{\~{}}kratt}}} 

\newbox\aaa
\setbox\aaa\hbox{\large$a$}
\title[Computation of the $a$-invariant of ladder determinantal rings]
{Computation of the {\box\aaa}-invariant of ladder determinantal rings}

\author[S. R. Ghorpade and C.~Krattenthaler]{\box\Adr}

\thanks{$\ast$ Research partially supported by the Indo-Russian project INT/RFBR/P-114 from the Department of Science \& Technology, Govt. of India and the  IRCC Award grant 12IRAWD009 from IIT Bombay}

\thanks{$^\dagger$Research partially supported by the Austrian
Science Foundation FWF, grants Z130-N13 and S50-N15,
the latter in the framework of the Special Research Program
``Algorithmic and Enumerative Combinatorics"}

\begin{abstract} 
We solve the problem of effectively computing the $a$-invariant of ladder
determinantal rings. In the case of a one-sided ladder, we
provide a compact formula, while, for a large family of two-sided ladders,
we provide an algorithmic solution.
\end{abstract}



\maketitle

\section{Introduction}
\label{sec:Intr}

Ladder determinantal rings are rings of polynomials in variables
$X_{i,j}$, $0\le i\le A$, $0\le j\le B$, 
modulo ideals generated by certain minors formed from
these variables 
 (see Section~\ref{sec:prelim} for the precise
definition). Ladder determinantal rings 
arose originally in the study of singularities of Schubert varieties
of flag manifolds  
by Abhyankar \cite{AbhyAB}. His work showed that ladder
determinantal rings are a  
natural generalization of determinantal rings corresponding to
classical determinantal ideals, and that they possess several nice
properties; for example, these rings are integral domains and are
rational in the sense that the quotient field is a purely
transcendental extension of the ground field. See also Narasimhan
\cite{NaraAA} for the former result. Ladder determinantal rings were
further studied by 
Herzog and Trung \cite{HeTrAA} who proved that these rings are
Cohen--Macaulay, using an explicit determination of the Gr\"obner basis
of the corresponding ladder determinantal ideal, and then showing that
the simplicial complex associated to its initial ideal is shellable.  
 Abhyankar and Kulkarni \cite{AbKuAC} have shown that the Hilbert
 function of ladder determinantal rings coincides with the Hilbert
 polynomial at all nonnegative integers.   
For more work on ladder determinantal rings, see
\cite{ConcAB, 
CoHeAB, GhorAC, 
GhorAF, KnMiAA,
KrPrAA, KrRuAA, KulkAD,
Wang}.

The purpose of the present paper is to provide methods for computing
the so-called $a$-invariant of ladder determinantal rings.
The $a$-invariant $a(R)$ is an important quantity associated with
a 
standard graded Cohen--Macaulay algebra $R$ over a field. It was introduced by 
Goto and Watanabe \cite{GoWaAA} as the negative of the least degree
of a generator of the graded canonical module of $R$. 
See \cite[p.~48]{ConcAD} for a summary of its various implications. 
In particular, it is argued there that it follows from work of
Stanley \cite{StanDA} that $a(R)=s-d$, given that the Hilbert
series of $R$ has the form $H(t)/(1-t)^d$, where 
$H(t)\in \Z[t]$ with $H(1)\ne 0$ and  $d$ is the Krull
dimension of $R$, while $s$ is the degree of 
$H(t)$. It is a classical result of Gr\"abe \cite{GraeAA} that, 
if $X = \left(X_{i,j}\right)$ 
is an $(A+1)\times (B+1)$matrix of variables, and $R$ the quotient of
the corresponding polynomial ring by the ideal generated by all
$(n+1)\times (n+1)$ minors of $X$, then  $a(R) = -\max\{A+1, B+1\}
n$. This result has been extended to weighted determinantal ideals and
Pfaffian ideals by Bruns and Herzog \cite{BrHeAA} and to ideals
cogenerated by a minor (and thus generated by minors of different 
sizes) of a rectangular matrix by Conca \cite{ConcAD} (see also
\cite[Theorem~4]{GhorAD}). The most  
general result appears to be that  of Conca \cite{ConcAD} on
determinantal rings (without ladder restriction). The case of ladder
determinantal rings appears to have been open and we take it up in
this paper.

Our first main result, consisting of 
Theorem~\ref{thm:1} and Corollary~\ref{cor:2}, provides a formula for
the $a$-invariant of one-sided ladder determinantal rings. It does not
reduce to Conca's formula in the special case where there is no
ladder restriction. Even in that case, our formula is simpler, as is
the proof of our formula. To explain the difference:
our proof follows Conca's in its first step, consisting of a
reduction of the problem to a problem of finding the largest set of
integer points in the plane satisfying certain properties (here, this
is hidden in the proof of Theorem~\ref{thm:3}; see 
\cite[Theorem~3.1]{RubeAC}), 
but differs fundamentally
from there on. While we translate these
point sets into families of non-intersecting lattice paths
(see Theorem~\ref{thm:3}), Conca
uses a version of the Robinson--Schensted--Knuth correspondence
in order to translate the point sets into pairs of semistandard
tableaux. As a matter of fact, the required analysis of the families of
non-intersecting lattice paths is much simpler than the corresponding
analysis of the pairs of tableaux. Moreover, the tableau approach
does not work in the presence of the ladder restriction.

Our second main result, consisting of 
Theorem~\ref{thm:2} and Corollary~\ref{cor:4}, provides an
algorithm for computing the $a$-invariant for a large family of 
two-sided ladder
determinantal rings. The idea behind this algorithm stems from
our result for the one-sided ladder case in Theorem~\ref{thm:1}.

The next section gives all necessary definitions and provides 
relevant background. In particular, it explains how the computation
of the $a$-invariant of ladder determinantal rings can be transformed
into the problem of counting non-intersecting lattice paths in 
ladder-shaped regions with a maximal total number of NE-turns,
see Theorem~\ref{thm:3}. In Section~\ref{sec:aux1} we then
solve the latter problem for {\it one-sided\/} ladder regions. The
resulting formula for the $a$-invariant of one-sided ladder
determinantal rings is presented in Section~\ref{sec:main1}.
The purpose of Section~\ref{sec:aux2} is to solve the problem
of counting non-intersecting lattice paths with a maximal total
number of NE-turns in {\it two-sided\/} ladder regions.
The corresponding result for the $a$-invariant of two-sided
ladder determinantal rings, which assumes a mild restriction
on the involved ladder region, is presented in Section~\ref{sec:main2}.

\section{Preliminaries}
\label{sec:prelim}

We start by recalling the definition of a ladder determinantal
ring. Let $K$ be a field and 
$X=(X_{i,j})_{0\le i\le A,\ 0\le j\le B}$ be an $(A+1)\times (B+1)$ 
matrix whose entries are independent indeterminates over $K$. Let
$Y=(Y_{i,j})_{0\le i\le A,\ 0\le j\le B}$ 
be another $(A+1)\times (B+1)$ matrix with the property that 
$Y_{i,j}=X_{i,j}$
or 0, and if $Y_{i,j}=X_{i,j}$ and $Y_{i'j'}=X_{i'j'}$, where $i\le i'$
and $j\le j'$, then $Y_{s,t}=X_{s,t}$ for all $s,t$ with $i\le s\le i'$
and $j\le t\le j'$. 
An example of such a matrix $Y$, 
with $A=15$ and $B=13$, is displayed in Figure~\ref{fig:4}.
Such a ``submatrix" $Y$ of $X$ is called a \emph{ladder}. This
terminology is motivated by the identification of such 
a matrix $Y$ with the set of all points $(j,A-i)$ in the
plane for which
$Y_{i,j}=X_{i,j}$. 
For example, the set of all such points for
the special matrix in Figure~\ref{fig:4} is shown in
Figure~\ref{fig:5}. (It should be apparent from comparison of
Figures~\ref{fig:4} and \ref{fig:5}
that the reason for taking $(j,A-i)$ instead of $(i,j)$ is to take
care of the difference in ``orientation" of row and column indexing of
a matrix versus coordinates in the plane.) 
In general, this set of
points looks like a (two-sided) ladder-shaped region. If, on the other
hand, we have either $Y_{0,0}=X_{0,0}$ or $Y_{a,b}=X_{a,b}$ then we
call $Y$ a {\it one-sided\/} ladder. In the first case we call $Y$ a
\emph{lower ladder}, in the second an \emph{upper ladder}.  Thus, the
matrix in Figure~\ref{fig:5a} is an upper ladder region (i.e.,
corresponds to a matrix $Y$ which is an upper ladder).

\begin{figure}[h]
\tiny
$$\Einheit.7cm
\left(
\hbox{\hskip.4cm}
\Label\o{ X\!\!_{1\!5,0}}(0,-8)
\Label\o{ X\!\!_{1\!5,1}}(1,-8)
\Label\o{ X\!\!_{1\!5,2}}(2,-8)
\Label\o{ X\!\!_{1\!5,3}}(3,-8)
\Label\o{ X\!\!_{1\!5,4}}(4,-8)
\Label\o{ X\!\!_{1\!5,5}}(5,-8)
\Label\o{ X\!\!_{1\!5,6}}(6,-8)
\Label\o{ X\!\!_{1\!5,7}}(7,-8)
\Label\o{ X\!\!_{1\!5,8}}(8,-8)
\Label\o{ 0}(9,-8)
\Label\o{ 0}(10,-8)
\Label\o{ 0}(11,-8)
\Label\o{ 0}(12,-8)
\Label\o{ 0}(13,-8)
\Label\o{ X\!\!_{1\!4,0}}(0,-7)
\Label\o{ X\!\!_{1\!4,1}}(1,-7)
\Label\o{ X\!\!_{1\!4,2}}(2,-7)
\Label\o{ X\!\!_{1\!4,3}}(3,-7)
\Label\o{ X\!\!_{1\!4,4}}(4,-7)
\Label\o{ X\!\!_{1\!4,5}}(5,-7)
\Label\o{ X\!\!_{1\!4,6}}(6,-7)
\Label\o{ X\!\!_{1\!4,7}}(7,-7)
\Label\o{ X\!\!_{1\!4,8}}(8,-7)
\Label\o{ 0}(9,-7)
\Label\o{ 0}(10,-7)
\Label\o{ 0}(11,-7)
\Label\o{ 0}(12,-7)
\Label\o{ 0}(13,-7)
\Label\o{ X\!\!_{1\!3,0}}(0,-6)
\Label\o{ X\!\!_{1\!3,1}}(1,-6)
\Label\o{ X\!\!_{1\!3,2}}(2,-6)
\Label\o{ X\!\!_{1\!3,3}}(3,-6)
\Label\o{ X\!\!_{1\!3,4}}(4,-6)
\Label\o{ X\!\!_{1\!3,5}}(5,-6)
\Label\o{ X\!\!_{1\!3,6}}(6,-6)
\Label\o{ X\!\!_{1\!3,7}}(7,-6)
\Label\o{ X\!\!_{1\!3,8}}(8,-6)
\Label\o{ X\!\!_{1\!3,9}}(9,-6)
\Label\o{ X\!\!_{1\!3,1\!0}}(10,-6)
\Label\o{ 0}(11,-6)
\Label\o{ 0}(12,-6)
\Label\o{ 0}(13,-6)
\Label\o{ X\!\!_{1\!2,0}}(0,-5)
\Label\o{ X\!\!_{1\!2,1}}(1,-5)
\Label\o{ X\!\!_{1\!2,2}}(2,-5)
\Label\o{ X\!\!_{1\!2,3}}(3,-5)
\Label\o{ X\!\!_{1\!2,4}}(4,-5)
\Label\o{ X\!\!_{1\!2,5}}(5,-5)
\Label\o{ X\!\!_{1\!2,6}}(6,-5)
\Label\o{ X\!\!_{1\!2,7}}(7,-5)
\Label\o{ X\!\!_{1\!2,8}}(8,-5)
\Label\o{ X\!\!_{1\!2,9}}(9,-5)
\Label\o{ X\!\!_{1\!2,1\!0}}(10,-5)
\Label\o{ 0}(11,-5)
\Label\o{ 0}(12,-5)
\Label\o{ 0}(13,-5)
\Label\o{ X\!\!_{1\!1,0}}(0,-4)
\Label\o{ X\!\!_{1\!1,1}}(1,-4)
\Label\o{ X\!\!_{1\!1,2}}(2,-4)
\Label\o{ X\!\!_{1\!1,3}}(3,-4)
\Label\o{ X\!\!_{1\!1,4}}(4,-4)
\Label\o{ X\!\!_{1\!1,5}}(5,-4)
\Label\o{ X\!\!_{1\!1,6}}(6,-4)
\Label\o{ X\!\!_{1\!1,7}}(7,-4)
\Label\o{ X\!\!_{1\!1,8}}(8,-4)
\Label\o{ X\!\!_{1\!1,9}}(9,-4)
\Label\o{ X\!\!_{1\!1,1\!0}}(10,-4)
\Label\o{ X\!\!_{1\!1,1\!1}}(11,-4)
\Label\o{ 0}(12,-4)
\Label\o{ 0}(13,-4)
\Label\o{ X\!\!_{1\!0,0}}(0,-3)
\Label\o{ X\!\!_{1\!0,1}}(1,-3)
\Label\o{ X\!\!_{1\!0,2}}(2,-3)
\Label\o{ X\!\!_{1\!0,3}}(3,-3)
\Label\o{ X\!\!_{1\!0,4}}(4,-3)
\Label\o{ X\!\!_{1\!0,5}}(5,-3)
\Label\o{ X\!\!_{1\!0,6}}(6,-3)
\Label\o{ X\!\!_{1\!0,7}}(7,-3)
\Label\o{ X\!\!_{1\!0,8}}(8,-3)
\Label\o{ X\!\!_{1\!0,9}}(9,-3)
\Label\o{ X\!\!_{1\!0,1\!0}}(10,-3)
\Label\o{ X\!\!_{1\!0,1\!1}}(11,-3)
\Label\o{ X\!\!_{1\!0,1\!2}}(12,-3)
\Label\o{ 0}(13,-3)
\Label\o{ X\!\!_{9,0}}(0,-2)
\Label\o{ X\!\!_{9,1}}(1,-2)
\Label\o{ X\!\!_{9,2}}(2,-2)
\Label\o{ X\!\!_{9,3}}(3,-2)
\Label\o{ X\!\!_{9,4}}(4,-2)
\Label\o{ X\!\!_{9,5}}(5,-2)
\Label\o{ X\!\!_{9,6}}(6,-2)
\Label\o{ X\!\!_{9,7}}(7,-2)
\Label\o{ X\!\!_{9,8}}(8,-2)
\Label\o{ X\!\!_{9,9}}(9,-2)
\Label\o{ X\!\!_{9,1\!0}}(10,-2)
\Label\o{ X\!\!_{9,1\!1}}(11,-2)
\Label\o{ X\!\!_{9,1\!2}}(12,-2)
\Label\o{ X\!\!_{9,1\!3}}(13,-2)
\Label\o{ 0}(0,-1)
\Label\o{ 0}(1,-1)
\Label\o{ 0}(2,-1)
\Label\o{ 0}(3,-1)
\Label\o{ X\!\!_{8,4}}(4,-1)
\Label\o{ X\!\!_{8,5}}(5,-1)
\Label\o{ X\!\!_{8,6}}(6,-1)
\Label\o{ X\!\!_{8,7}}(7,-1)
\Label\o{ X\!\!_{8,8}}(8,-1)
\Label\o{ X\!\!_{8,9}}(9,-1)
\Label\o{ X\!\!_{8,1\!0}}(10,-1)
\Label\o{ X\!\!_{8,1\!1}}(11,-1)
\Label\o{ X\!\!_{8,1\!2}}(12,-1)
\Label\o{ X\!\!_{8,1\!3}}(13,-1)
\Label\o{ 0}(0,0)
\Label\o{ 0}(1,0)
\Label\o{ 0}(2,0)
\Label\o{ 0}(3,0)
\Label\o{ X\!\!_{7,4}}(4,0)
\Label\o{ X\!\!_{7,5}}(5,0)
\Label\o{ X\!\!_{7,6}}(6,0)
\Label\o{ X\!\!_{7,7}}(7,0)
\Label\o{ X\!\!_{7,8}}(8,0)
\Label\o{ X\!\!_{7,9}}(9,0)
\Label\o{ X\!\!_{7,1\!0}}(10,0)
\Label\o{ X\!\!_{7,1\!1}}(11,0)
\Label\o{ X\!\!_{7,1\!2}}(12,0)
\Label\o{ X\!\!_{7,1\!3}}(13,0)
\Label\o{ 0}(0,1)
\Label\o{ 0}(1,1)
\Label\o{ 0}(2,1)
\Label\o{ 0}(3,1)
\Label\o{ X\!\!_{6,4}}(4,1)
\Label\o{ X\!\!_{6,5}}(5,1)
\Label\o{ X\!\!_{6,6}}(6,1)
\Label\o{ X\!\!_{6,7}}(7,1)
\Label\o{ X\!\!_{6,8}}(8,1)
\Label\o{ X\!\!_{6,9}}(9,1)
\Label\o{ X\!\!_{6,1\!0}}(10,1)
\Label\o{ X\!\!_{6,1\!1}}(11,1)
\Label\o{ X\!\!_{6,1\!2}}(12,1)
\Label\o{ X\!\!_{6,1\!3}}(13,1)
\Label\o{ 0}(0,2)
\Label\o{ 0}(1,2)
\Label\o{ 0}(2,2)
\Label\o{ 0}(3,2)
\Label\o{ 0}(4,2)
\Label\o{ X\!\!_{5,5}}(5,2)
\Label\o{ X\!\!_{5,6}}(6,2)
\Label\o{ X\!\!_{5,7}}(7,2)
\Label\o{ X\!\!_{5,8}}(8,2)
\Label\o{ X\!\!_{5,9}}(9,2)
\Label\o{ X\!\!_{5,1\!0}}(10,2)
\Label\o{ X\!\!_{5,1\!1}}(11,2)
\Label\o{ X\!\!_{5,1\!2}}(12,2)
\Label\o{ X\!\!_{5,1\!3}}(13,2)
\Label\o{ 0}(0,3)
\Label\o{ 0}(1,3)
\Label\o{ 0}(2,3)
\Label\o{ 0}(3,3)
\Label\o{ 0}(4,3)
\Label\o{ 0}(5,3)
\Label\o{ X\!\!_{4,6}}(6,3)
\Label\o{ X\!\!_{4,7}}(7,3)
\Label\o{ X\!\!_{4,8}}(8,3)
\Label\o{ X\!\!_{4,9}}(9,3)
\Label\o{ X\!\!_{4,1\!0}}(10,3)
\Label\o{ X\!\!_{4,1\!1}}(11,3)
\Label\o{ X\!\!_{4,1\!2}}(12,3)
\Label\o{ X\!\!_{4,1\!3}}(13,3)
\Label\o{ 0}(0,4)
\Label\o{ 0}(1,4)
\Label\o{ 0}(2,4)
\Label\o{ 0}(3,4)
\Label\o{ 0}(4,4)
\Label\o{ 0}(5,4)
\Label\o{ 0}(6,4)
\Label\o{ X\!\!_{3,7}}(7,4)
\Label\o{ X\!\!_{3,8}}(8,4)
\Label\o{ X\!\!_{3,9}}(9,4)
\Label\o{ X\!\!_{3,1\!0}}(10,4)
\Label\o{ X\!\!_{3,1\!1}}(11,4)
\Label\o{ X\!\!_{3,1\!2}}(12,4)
\Label\o{ X\!\!_{3,1\!3}}(13,4)
\Label\o{ 0}(0,5)
\Label\o{ 0}(1,5)
\Label\o{ 0}(2,5)
\Label\o{ 0}(3,5)
\Label\o{ 0}(4,5)
\Label\o{ 0}(5,5)
\Label\o{ 0}(6,5)
\Label\o{ 0}(7,5)
\Label\o{ X\!\!_{2,8}}(8,5)
\Label\o{ X\!\!_{2,9}}(9,5)
\Label\o{ X\!\!_{2,1\!0}}(10,5)
\Label\o{ X\!\!_{2,1\!1}}(11,5)
\Label\o{ X\!\!_{2,1\!2}}(12,5)
\Label\o{ X\!\!_{2,1\!3}}(13,5)
\Label\o{ 0}(0,6)
\Label\o{ 0}(1,6)
\Label\o{ 0}(2,6)
\Label\o{ 0}(3,6)
\Label\o{ 0}(4,6)
\Label\o{ 0}(5,6)
\Label\o{ 0}(6,6)
\Label\o{ 0}(7,6)
\Label\o{ X\!\!_{1,8}}(8,6)
\Label\o{ X\!\!_{1,9}}(9,6)
\Label\o{ X\!\!_{1,1\!0}}(10,6)
\Label\o{ X\!\!_{1,1\!1}}(11,6)
\Label\o{ X\!\!_{1,1\!2}}(12,6)
\Label\o{ X\!\!_{1,1\!3}}(13,6)
\Label\o{ 0}(0,7)
\Label\o{ 0}(1,7)
\Label\o{ 0}(2,7)
\Label\o{ 0}(3,7)
\Label\o{ 0}(4,7)
\Label\o{ 0}(5,7)
\Label\o{ 0}(6,7)
\Label\o{ 0}(7,7)
\Label\o{ X\!\!_{0,8}}(8,7)
\Label\o{ X\!\!_{0,9}}(9,7)
\Label\o{ X\!\!_{0,1\!0}}(10,7)
\Label\o{ X\!\!_{0,1\!1}}(11,7)
\Label\o{ X\!\!_{0,1\!2}}(12,7)
\Label\o{ X\!\!_{0,1\!3}}(13,7)
\hskip9.5cm
\right)
$$
\caption{\small A two-sided ladder}
\label{fig:4}
\end{figure}

\vskip10pt
\begin{figure}
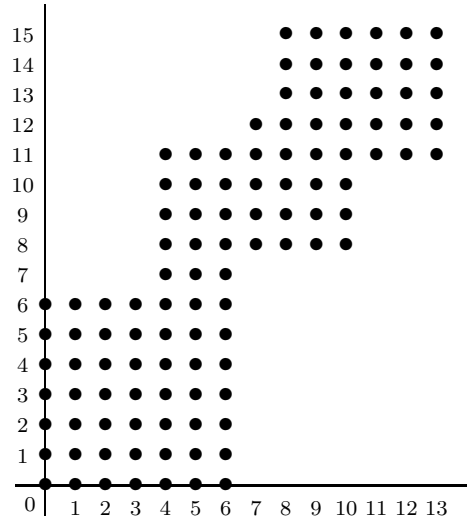

\tiny
$$\Einheit.4cm
\Koordinatenachsen(14,16)(0,0)
\DuennPunkt(0,0)
\DuennPunkt(1,0)
\DuennPunkt(2,0)
\DuennPunkt(3,0)
\DuennPunkt(4,0)
\DuennPunkt(5,0)
\DuennPunkt(6,0)
%
\DuennPunkt(0,1)
\DuennPunkt(1,1)
\DuennPunkt(2,1)
\DuennPunkt(3,1)
\DuennPunkt(4,1)
\DuennPunkt(5,1)
\DuennPunkt(6,1)
%
\DuennPunkt(0,2)
\DuennPunkt(1,2)
\DuennPunkt(2,2)
\DuennPunkt(3,2)
\DuennPunkt(4,2)
\DuennPunkt(5,2)
\DuennPunkt(6,2)
%
\DuennPunkt(0,3)
\DuennPunkt(1,3)
\DuennPunkt(2,3)
\DuennPunkt(3,3)
\DuennPunkt(4,3)
\DuennPunkt(5,3)
\DuennPunkt(6,3)
%
\DuennPunkt(0,4)
\DuennPunkt(1,4)
\DuennPunkt(2,4)
\DuennPunkt(3,4)
\DuennPunkt(4,4)
\DuennPunkt(5,4)
\DuennPunkt(6,4)
%
\DuennPunkt(0,5)
\DuennPunkt(1,5)
\DuennPunkt(2,5)
\DuennPunkt(3,5)
\DuennPunkt(4,5)
\DuennPunkt(5,5)
\DuennPunkt(6,5)
%
\DuennPunkt(0,6)
\DuennPunkt(1,6)
\DuennPunkt(2,6)
\DuennPunkt(3,6)
\DuennPunkt(4,6)
\DuennPunkt(5,6)
\DuennPunkt(6,6)
%
\DuennPunkt(4,7)
\DuennPunkt(5,7)
\DuennPunkt(6,7)
%
\DuennPunkt(4,8)
\DuennPunkt(5,8)
\DuennPunkt(6,8)
\DuennPunkt(7,8)
\DuennPunkt(8,8)
\DuennPunkt(9,8)
\DuennPunkt(10,8)
%
\DuennPunkt(4,9)
\DuennPunkt(5,9)
\DuennPunkt(6,9)
\DuennPunkt(7,9)
\DuennPunkt(8,9)
\DuennPunkt(9,9)
\DuennPunkt(10,9)
%
\DuennPunkt(4,10)
\DuennPunkt(5,10)
\DuennPunkt(6,10)
\DuennPunkt(7,10)
\DuennPunkt(8,10)
\DuennPunkt(9,10)
\DuennPunkt(10,10)
%
\DuennPunkt(4,11)
\DuennPunkt(5,11)
\DuennPunkt(6,11)
\DuennPunkt(7,11)
\DuennPunkt(8,11)
\DuennPunkt(9,11)
\DuennPunkt(10,11)
\DuennPunkt(11,11)
\DuennPunkt(12,11)
\DuennPunkt(13,11)
\DuennPunkt(7,12)
\DuennPunkt(8,12)
\DuennPunkt(9,12)
\DuennPunkt(10,12)
\DuennPunkt(11,12)
\DuennPunkt(12,12)
\DuennPunkt(13,12)
\DuennPunkt(8,13)
\DuennPunkt(9,13)
\DuennPunkt(10,13)
\DuennPunkt(11,13)
\DuennPunkt(12,13)
\DuennPunkt(13,13)
\DuennPunkt(8,14)
\DuennPunkt(9,14)
\DuennPunkt(10,14)
\DuennPunkt(11,14)
\DuennPunkt(12,14)
\DuennPunkt(13,14)
\DuennPunkt(8,15)
\DuennPunkt(9,15)
\DuennPunkt(10,15)
\DuennPunkt(11,15)
\DuennPunkt(12,15)
\DuennPunkt(13,15)
\Label\lu{0}(0,0)
\Label\u{ 1}(1,0)
\Label\u{ 2}(2,0)
\Label\u{ 3}(3,0)
\Label\u{ 4}(4,0)
\Label\u{ 5}(5,0)
\Label\u{ 6}(6,0)
\Label\u{ 7}(7,0)
\Label\u{ 8}(8,0)
\Label\u{ 9}(9,0)
\Label\u{ 10}(10,0)
\Label\u{ 11}(11,0)
\Label\u{ 12}(12,0)
\Label\u{ 13}(13,0)
\Label\l{ 1}(0,1)
\Label\l{ 2}(0,2)
\Label\l{ 3}(0,3)
\Label\l{ 4}(0,4)
\Label\l{ 5}(0,5)
\Label\l{ 6}(0,6)
\Label\l{ 7}(0,7)
\Label\l{ 8}(0,8)
\Label\l{ 9}(0,9)
\Label\l{ 10}(0,10)
\Label\l{ 11}(0,11)
\Label\l{ 12}(0,12)
\Label\l{ 13}(0,13)
\Label\l{ 14}(0,14)
\Label\l{ 15}(0,15)
\hskip5.0cm
$$
\caption{\small A two-sided ladder region}
\label{fig:5}
\end{figure}

\vskip10pt
\begin{figure}[h]
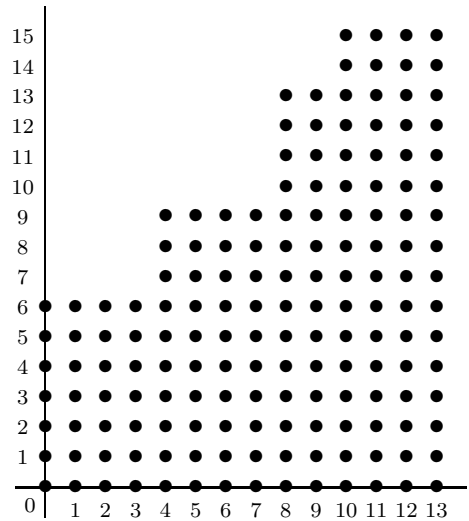

\tiny
$$\Einheit.4cm
\Koordinatenachsen(14,16)(0,0)
\DuennPunkt(0,0)
\DuennPunkt(1,0)
\DuennPunkt(2,0)
\DuennPunkt(3,0)
\DuennPunkt(4,0)
\DuennPunkt(5,0)
\DuennPunkt(6,0)
\DuennPunkt(7,0)
\DuennPunkt(8,0)
\DuennPunkt(9,0)
\DuennPunkt(10,0)
\DuennPunkt(11,0)
\DuennPunkt(12,0)
\DuennPunkt(13,0)
\DuennPunkt(0,1)
\DuennPunkt(1,1)
\DuennPunkt(2,1)
\DuennPunkt(3,1)
\DuennPunkt(4,1)
\DuennPunkt(5,1)
\DuennPunkt(6,1)
\DuennPunkt(7,1)
\DuennPunkt(8,1)
\DuennPunkt(9,1)
\DuennPunkt(10,1)
\DuennPunkt(11,1)
\DuennPunkt(12,1)
\DuennPunkt(13,1)
\DuennPunkt(0,2)
\DuennPunkt(1,2)
\DuennPunkt(2,2)
\DuennPunkt(3,2)
\DuennPunkt(4,2)
\DuennPunkt(5,2)
\DuennPunkt(6,2)
\DuennPunkt(7,2)
\DuennPunkt(8,2)
\DuennPunkt(9,2)
\DuennPunkt(10,2)
\DuennPunkt(11,2)
\DuennPunkt(12,2)
\DuennPunkt(13,2)
\DuennPunkt(0,3)
\DuennPunkt(1,3)
\DuennPunkt(2,3)
\DuennPunkt(3,3)
\DuennPunkt(4,3)
\DuennPunkt(5,3)
\DuennPunkt(6,3)
\DuennPunkt(7,3)
\DuennPunkt(8,3)
\DuennPunkt(9,3)
\DuennPunkt(10,3)
\DuennPunkt(11,3)
\DuennPunkt(12,3)
\DuennPunkt(13,3)
\DuennPunkt(0,4)
\DuennPunkt(1,4)
\DuennPunkt(2,4)
\DuennPunkt(3,4)
\DuennPunkt(4,4)
\DuennPunkt(5,4)
\DuennPunkt(6,4)
\DuennPunkt(7,4)
\DuennPunkt(8,4)
\DuennPunkt(9,4)
\DuennPunkt(10,4)
\DuennPunkt(11,4)
\DuennPunkt(12,4)
\DuennPunkt(13,4)
\DuennPunkt(0,5)
\DuennPunkt(1,5)
\DuennPunkt(2,5)
\DuennPunkt(3,5)
\DuennPunkt(4,5)
\DuennPunkt(5,5)
\DuennPunkt(6,5)
\DuennPunkt(7,5)
\DuennPunkt(8,5)
\DuennPunkt(9,5)
\DuennPunkt(10,5)
\DuennPunkt(11,5)
\DuennPunkt(12,5)
\DuennPunkt(13,5)
\DuennPunkt(0,6)
\DuennPunkt(1,6)
\DuennPunkt(2,6)
\DuennPunkt(3,6)
\DuennPunkt(4,6)
\DuennPunkt(5,6)
\DuennPunkt(6,6)
\DuennPunkt(7,6)
\DuennPunkt(8,6)
\DuennPunkt(9,6)
\DuennPunkt(10,6)
\DuennPunkt(11,6)
\DuennPunkt(12,6)
\DuennPunkt(13,6)
\DuennPunkt(4,7)
\DuennPunkt(5,7)
\DuennPunkt(6,7)
\DuennPunkt(7,7)
\DuennPunkt(8,7)
\DuennPunkt(9,7)
\DuennPunkt(10,7)
\DuennPunkt(11,7)
\DuennPunkt(12,7)
\DuennPunkt(13,7)
\DuennPunkt(4,8)
\DuennPunkt(5,8)
\DuennPunkt(6,8)
\DuennPunkt(7,8)
\DuennPunkt(8,8)
\DuennPunkt(9,8)
\DuennPunkt(10,8)
\DuennPunkt(11,8)
\DuennPunkt(12,8)
\DuennPunkt(13,8)
\DuennPunkt(4,9)
\DuennPunkt(5,9)
\DuennPunkt(6,9)
\DuennPunkt(7,9)
\DuennPunkt(8,9)
\DuennPunkt(9,9)
\DuennPunkt(10,9)
\DuennPunkt(11,9)
\DuennPunkt(12,9)
\DuennPunkt(13,9)
%
\DuennPunkt(8,10)
\DuennPunkt(9,10)
\DuennPunkt(10,10)
\DuennPunkt(11,10)
\DuennPunkt(12,10)
\DuennPunkt(13,10)
%
\DuennPunkt(8,11)
\DuennPunkt(9,11)
\DuennPunkt(10,11)
\DuennPunkt(11,11)
\DuennPunkt(12,11)
\DuennPunkt(13,11)
%
\DuennPunkt(8,12)
\DuennPunkt(9,12)
\DuennPunkt(10,12)
\DuennPunkt(11,12)
\DuennPunkt(12,12)
\DuennPunkt(13,12)
\DuennPunkt(8,13)
\DuennPunkt(9,13)
\DuennPunkt(10,13)
\DuennPunkt(11,13)
\DuennPunkt(12,13)
\DuennPunkt(13,13)
%
\DuennPunkt(10,14)
\DuennPunkt(11,14)
\DuennPunkt(12,14)
\DuennPunkt(13,14)
%
\DuennPunkt(10,15)
\DuennPunkt(11,15)
\DuennPunkt(12,15)
\DuennPunkt(13,15)
\Label\lu{0}(0,0)
\Label\u{ 1}(1,0)
\Label\u{ 2}(2,0)
\Label\u{ 3}(3,0)
\Label\u{ 4}(4,0)
\Label\u{ 5}(5,0)
\Label\u{ 6}(6,0)
\Label\u{ 7}(7,0)
\Label\u{ 8}(8,0)
\Label\u{ 9}(9,0)
\Label\u{ 10}(10,0)
\Label\u{ 11}(11,0)
\Label\u{ 12}(12,0)
\Label\u{ 13}(13,0)
\Label\l{ 1}(0,1)
\Label\l{ 2}(0,2)
\Label\l{ 3}(0,3)
\Label\l{ 4}(0,4)
\Label\l{ 5}(0,5)
\Label\l{ 6}(0,6)
\Label\l{ 7}(0,7)
\Label\l{ 8}(0,8)
\Label\l{ 9}(0,9)
\Label\l{ 10}(0,10)
\Label\l{ 11}(0,11)
\Label\l{ 12}(0,12)
\Label\l{ 13}(0,13)
\Label\l{ 14}(0,14)
\Label\l{ 15}(0,15)
\hskip5.0cm
$$
\caption{\small An upper ladder region}
\label{fig:5a}
\end{figure}

Now fix a ``bivector" $M=[u_1,u_2,\dots,u_n\mid v_1,v_2,\dots,v_n]$ of
positive integers with $u_1<u_2<\dots<u_n \le A+1$ and
$v_1<v_2<\dots<v_n\le B+1$. 
Let $K[Y]$ denote the ring of all polynomials over 
the field $K$
in the $Y_{i,j}$'s, $0\le i\le A$, $0\le j\le B$, and let $I_M(Y)$
be the ideal 
of $K[Y]$ 
generated by those $t\times t$ minors of $Y$
that contain only nonzero entries, whose rows form a subset of the
last $u_t-1$ rows, or whose columns form a subset of the last $v_t-1$
columns, $t=1,2,\dots,n+1$. Here, by convention, $u_{n+1}$ is set
equal to $A+2$, and $v_{n+1}$ is set equal to $B+2$.  (Thus, for
$t=n+1$ the rows and columns of minors are unrestricted.)  The ideal
$I_M(Y)$ is called a {\it ladder determinantal ideal cogenerated by
the minor $M$}. (That one speaks of `the minor $M$' has its
explanation in the identification of the bivector $M$ with a
particular minor of $Y$, cf\@. \cite[Sec.~2]{HeTrAA}. It should be
pointed out that our conventions here deviate slightly from the ones
in \cite{HeTrAA}. In particular, we defined the ideal $I_M(Y)$ by
restricting rows and columns of minors to a certain number of {\it
last\/} rows or columns, while in \cite{HeTrAA} it is {\it first\/}
rows, respectively columns. Clearly, a rotation of the matrix by $180^\circ$
transforms one convention into the other.) 
The associated {\it ladder
determinantal ring cogenerated by $M$\/} is $R_M(Y):=K[Y]/I_M(Y)$. (We
point out that the definition of ladder is more general in 
\cite{AbhyAB,AbKuAC,ConcAB,HeTrAA}. However, there is in effect no loss of
generality since the ladders of \cite{AbhyAB,AbKuAC,ConcAB,HeTrAA} 
can always be reduced to our definition by discarding
superfluous 0's.)

Generalising results of Abhyankar and Kulkarni
\cite{AbhyAB,AbKuAC}, Herzog and Trung \cite{HeTrAA} provided a way
to express the
Hilbert series of the ladder determinantal ring $R_M(Y)$ in 
combinatorial terms. Before we can state the corresponding result, 
as derived by Rubey \cite{RubeAC}, we
need to introduce a few more terms.

When we say {\it lattice path\/} we always mean
a lattice path in the plane consisting of unit horizontal and vertical
steps in the positive direction. In other words, a \emph{lattice path}
is a finite sequence $A_0, A_1, \dots, A_m$ of points in $\Z^2$ such that
$A_i - A_{i-1} = (1,0)$ or $(0,1)$ for all $i=1, \dots , m$. Such a
sequence is sometimes called a lattice path from $A_0$ to $A_m$. In
case the successive differences $A_i - A_{i-1}$ always alternate
between $(1,0)$ or $(0,1)$, then we refer to it as a \emph{zig-zag
  path}. See Figure~\ref{fig:6} for an illustration of a lattice path
from $(1, -1)$ to $(6,6)$. This is not a zig-zag path, but its part
from $(4,3)$ to $(6,5)$ is.  

\begin{figure}[h]
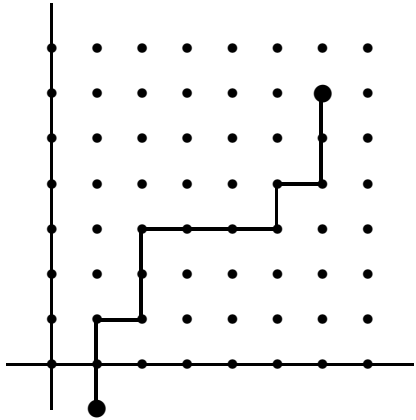

$$\Koordinatenachsen(8,8)(0,0)
\Gitter(8,8)(0,0)
\Pfad(1,-1),221221112122\endPfad
\DickPunkt(1,-1)
\DickPunkt(6,6)
\hskip4cm
$$
\caption{\small A lattice path}
\label{fig:6}
\end{figure}

\noindent
A family $(P_1,P_2,\dots,P_n)$ of
lattice paths $P_{i}$, $i=1,2,\dots,n$, is said to be {\it
non-intersecting\/} if no two lattice paths of this family have a point
in common.  

A point in a lattice path $P$ which is the end point of a vertical
step and at the same time the starting point of a horizontal step will
be called a {\it north-east turn\/} ({\it NE-turn\/} for short) of the
lattice path $P$.  The NE-turns of the lattice path in
Figure~\ref{fig:6} are
$(1,1)$, $(2,3)$, and $(5,4)$. We write $\NE(P)$ for the number of
NE-turns of $P$. Also, given a family $\mathbf P=(P_1,P_2,\dots,P_n)$ of
paths $P_i$, we write $\NE(\mathbf P)$ for the number $\sum _{i=1}
^{n}\NE(P_i)$ of all NE-turns in the family.

We shall say that a lattice path
$P$ stays (or passes) \emph{weakly south-east} of a lattice point
$S=(s_1,s_2)$ if each point $(x,y)$ of $P$ satisfies 
either $x\ge s_1$ (``the point $(x,y)$ lies weakly east
of $S$") or $y\le s_2$ (``the point $(x,y)$ lies weakly south
of $S$"), or both.
Sometimes, this may be expressed by saying that the point
$S$ lies \emph{weakly north-west} of $P$. 
To stay (or pass) \emph{weakly north-east} of a lattice point 
and to stay (or pass) 
\emph{weakly north-west} of a lattice point will have an analogous meaning.

Finally, given any weight function $w$ defined on a finite set
$\mathcal M$ and taking values in a commutative ring, 
by the generating function $\GF(\mathcal M;w)$ we mean 
$\sum _{x\in\mathcal M}^{}w(x)$. 

We are now in the position to state the theorem which connects the
computation of the Hilbert series of a ladder determinantal ring
with the enumeration of (certain) non-intersecting lattice paths.
For a proof, the reader is referred to \cite[Theorem~3.1]{RubeAC}.

\begin{Theorem}
\label{thm:3}
Let $Y=(Y_{i,j})_{0\le i\le A,\ 0\le j\le B}$ be
a two-sided ladder, and let $L$ be the associated ladder region.
Let $M=[u_1,u_2,\dots,u_n\mid v_1,v_2,\dots,v_n]$
be a bivector of positive integers with 
$u_1<u_2<\dots<u_n$ and $v_1<v_2<\dots<v_n$.
Furthermore, let $A^{(i)}=(0,u_{n-i+1}-1)$ and $E^{(i)}=(B-v_{n-i+1}+1,A)$, 
$i=1,2,\dots,n$. 
Recursively, define the regions $L^{(i)}$, $i=n,n-1,\dots,1$,
by $L^{(n)}=L$ and
$$
L^{(i)}=\{(x,y)\in L^{(i+1)}:x\le E_1^{(i)},\ y\ge A_2^{(i)},
\text{ and\/ }(x+1,y-1)\in L^{(i+1)}\}.
$$
Finally, for $i=1,2,\dots,n$ let
$$
B^{(i)}=\{(x,y)\in L^{(i)}:(x+1,y-1)\notin L^{(i)}\},
$$
and let $d$ be the cardinality of\/ $\bigcup_{i=1}^n B^{(i)}$.

Then, under the assumption that all of the points $A^{(i)}$ and
$E^{(i)}$, $i=1,2,\dots,n$,
lie inside the ladder region $L$, the Hilbert series of the ladder determinantal
ring $R_M(Y)=K[Y]/I_M(Y)$ equals
\begin{equation}
\sum _{\ell=0} ^{\infty}\dim_K R_M(Y)_\ell\,z^\ell
=\frac {\GF(\P_L^+(\mathbf A\to \mathbf E);z^{\NE(.)})}
{(1-z)^{d}},
\label{e2.1}
\end{equation}
where $R_M(Y)_\ell$ denotes the homogeneous component of degree $\ell$
in $R_M(Y)$, and where
$\GF(\P_L^+(\mathbf A\to \mathbf E);z^{\NE(.)})$ denotes the generating
function $\sum _{\mathbf P} ^{}z^{\NE(\mathbf P)}$ for all families $\mathbf
P=(P_1,P_2,\dots,P_n)$ of non-intersecting lattice paths, $P_i$ running
from $A^{(i)}$ to $E^{(i)}$ with all its NE-turns lying in 
$L^{(i)}\backslash B^{(i)}$.
\end{Theorem}

\begin{Remarks} \label{rem:1}
(1)
The condition that all of the points $A^{(i)}$ and $E^{(i)}$
lie inside the ladder region $L$ restricts the choice of ladders. In
particular, for an upper ladder it means that 
$Y_{A-u_{n}+1,0}=X_{A-u_{n}+1,0}$ and
$Y_{0,B-v_{n}+1}=X_{0,B-v_{n}+1}$.
Still, one could prove an analogous result
even if this condition is dropped. In that case, however,
the points $A^{(i)}$ and $E^{(i)}$ have to be modified in
order to lie inside $L$ 
so as to make the right-hand side of
Formula \eqref{e2.1} meaningful. 

\medskip
(2)
The sets $B^{(i)}$, $i=1, 2,\dots , n$, can be visualized as being the
lower-right boundary 
of $L^{(i)}$. Viewed as a path, there are exactly 
$E^{(i)}_1 -A^{(i)}_1 +E^{(i)}_2 -A^{(i)}_2 +1$ lattice
points on $B^{(i)}$, but not all of them are necessarily in $L$
(see \cite[Figures~2 and~3]{RubeAC} for an example). However, if they are,
then
\begin{align*}
d&=
\sum _{i=1} ^{n}\big(E^{(i)}_1 -A^{(i)}_1 +E^{(i)}_2 -A^{(i)}_2
+1\big)\\
&=\sum _{i=1} ^{n}\big((B-v_{n-i+1}+1) +A-(u_{n-i+1}-1)
+1\big)\\
&=(A+B+3)n-
\sum _{i=1} ^{n}(u_i+v_i).
\end{align*}

\medskip
(3) In the case of a one-sided ladder, all the $B^{(i)}$'s are
completely contained in $L$ so that the above remark on $d$ applies.
Furthermore, if the one-sided ladder should be an upper ladder,
then it is easy to see that the technical condition in Theorem~\ref{thm:3}
involving the $L^{(i)}$'s and $B^{(i)}$'s reduces to the much simpler
(and much more intuitive) condition that all the $P_i$'s should
completely lie in~$L$.

\medskip
(4)
It should be observed that the condition imposed on the paths
$P_i$ that all of its NE-turns lie in $L^{(i)}$ (and, thus,
in~$L$) does not imply that $P_i$ lies completely in $L$
(namely, it may run below the lower boundary of $L$);
see \cite[Figure~4]{RubeAC} for an example. 
\end{Remarks}

If we combine the formula for the Hilbert series of $R_M(Y)$ in
Theorem~\ref{thm:3} with the observation made in the introduction 
on how to extract the
$a$-invariant out of such a formula, then we obtain immediately the
following corollary.

\begin{Corollary} \label{cor:1}
Under the assumptions and the notation of Theorem~{\em\ref{thm:3}}, the $a$-invariant of
the ladder determinantal ring $R_M(Y)=K[Y]/I_M(Y)$ is given by
$$
\deg\left(\GF(\P_L^+(\mathbf A\to
  \mathbf E);z^{\NE(.)})\right)-d.
$$
\end{Corollary}

Hence, if we want to express the $a$-invariant of $R_M(Y)$ in terms
of $M$ and the ladder $Y$, then we must determine the degree of the
polynomial $\GF(\P_L^+(\mathbf A\to \mathbf E);z^{\NE(.)})$. This
amounts to determining the maximum number of NE-turns a family
of non-intersecting lattice paths as described in Theorem~\ref{thm:3}
can attain. This is what we shall do in the following sections.

\section{How to achieve the maximum number of NE-turns: the one-sided
case}
\label{sec:aux1}

We start with the consideration of upper ladders, see Figure~\ref{fig:5a} for
an example. By Remark~\ref{rem:1}.(3), in that case we do not
have to worry about the
technical condition involving the $L^{(i)}$'s and the $B^{(i)}$'s
as long as we make sure that all the paths $P^{(i)}$ lie completely
in~$L$.

We begin with the task of maximising the number of NE-turns of a
single path in an upper ladder. In this and the following section, we formulate
the ladder restriction as the restriction that paths should stay
south-east of some given lattice points. Clearly, the restriction
imposed by an upper ladder can be formulated in that way: one chooses
the points $S_1,S_2,\dots$ in Lemmas~\ref{lem:1} and \ref{lem:2}
as the ``inwards" corners of the upper boundary of $L$, that is,
the elements $(x,y)\in L$ for which both $(x-1,y)$ and $(x,y+1)$
are in $L$ but $(x-1,y+1)$ is not. For example, the inwards corners
of the ladder region in Figure~\ref{fig:5a} are 
$(4,6)$,
$(8,9)$,
and $(10,13)$.

\begin{Lemma} \label{lem:1}
Let $A=(a,b)$, $B=(c,d)$, and $S_i=(x_i,y_i)$ be
lattice points with $a\le x_i\le c$ and $b\le y_i\le d$,
for $i=1,2,\dots,m$. The number of NE-turns of a lattice path from $A$ 
to $B$ which stays weakly south-east of $S_i$, for $i=1,2,\dots,m$, 
is at most
\begin{equation} \label{eq:1}
\min\!\big\{c-a,d-b,c-b-\max\{x_i-y_i:1\le i\le m\}\big\}.
\end{equation}
The maximum is for instance realised by the path which consists of a
zig-zag path which passes through
one of the points $S_j$ for which $x_j-y_j$ equals
$\max\{x_i-y_i:1\le i\le m\}$, 
supplemented by a straight horizontal piece at the beginning and 
a straight vertical piece at the end, as is necessary to connect $A$
with $B$; cf.\ Figure~{\em\ref{fig:2}}.
\end{Lemma}

\begin{figure}[h]
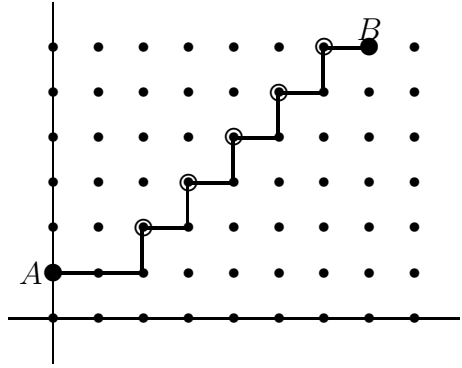

$$
\Gitter(9,7)(0,0)
\Koordinatenachsen(9,7)(0,0)
\Pfad(0,1),112121212121\endPfad
\DickPunkt(0,1)
\DickPunkt(7,6)
\Kreis(2,2)
\Kreis(3,3)
\Kreis(4,4)
\Kreis(5,5)
\Kreis(6,6)
\Label\l{A}(0,1)
\Label\o{B}(7,6)
\hskip5cm
$$
\caption{\small A lattice path with maximal number of NE-turns}
\label{fig:1}
\end{figure}

\begin{Remark} \label{rem:2}
The expression in \eqref{eq:1} could be further economised to
$$
c-b-\max\{x_i-y_i:0\le i\le m+1\},
$$
by including $(a,b)$ and $(c,d)$ in the restriction points $S_i$, that
is, by setting $S_0=(a,b)$ and $S_{m+1}=(c,d)$.
\end{Remark}

\begin{proof}[Proof of Lemma~\ref{lem:1}] 
We discuss the case where $d-b\le c-a$, the other case being
completely analogous.
If $m=0$, that is, if the set of points $S_j$ is empty, then the path
which starts with a straight horizontal piece from $A$ to $(b+c-d,b)$
and then continues with a zig-zag path until $B$ 
(starting with an up-step and
terminating with a right-step; see the example in
Figure~\ref{fig:1}, where $a=0$, $b=1$, $c=7$, and $d=6$) 
attains the maximal possible number of NE-turns
for paths between $A$ and $B$, namely~$d-b$.

\begin{figure}[h]
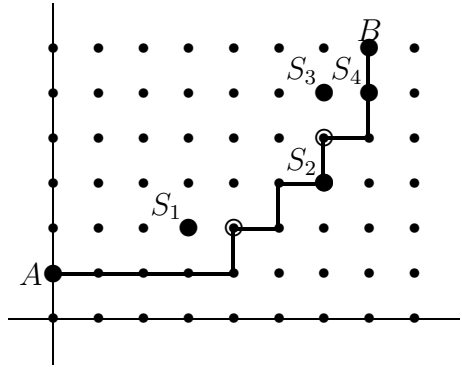

$$
\Gitter(9,7)(0,0)
\Koordinatenachsen(9,7)(0,0)
\Pfad(0,1),111121212122\endPfad
\DickPunkt(0,1)
\DickPunkt(7,6)
\DickPunkt(3,2)
\DickPunkt(6,3)
\DickPunkt(6,5)
\DickPunkt(7,5)
\Kreis(4,2)
\Kreis(6,3)
\Kreis(6,4)
\Label\l{A}(0,1)
\Label\o{B}(7,6)
\Label\lo{S_1}(3,2)
\Label\lo{S_2}(6,3)
\Label\lo{S_3}(6,5)\
\Label\lo{S_4}(7,5)
\hskip5cm
$$
\caption{\small A lattice path with maximal number of NE-turns 
south-east of the $S_i$'s}
\label{fig:2}
\end{figure}

Now let $m\ge1$. Clearly, as long as all points $S_j$ lie weakly 
north-west of the path constructed above (and exemplified in
Figure~\ref{fig:1}), the maximal possible number of NE-turns will
still equal $d-b$. This is 
well 
in accordance with \eqref{eq:1}.
In symbols, this case is characterised by the property that
$x_j-y_j\le c-d$ for all $j$.

On the other hand, let $S_j$ be a point for which $x_j-y_j$ equals
$\max\{x_i-y_i:1\le i\le m\}$ and $x_j-y_j>c-d$. 
It is then obvious (see also
the example in Figure~\ref{fig:2}, where $m=4$,
$A=(0,1)$,
$B=(7,6)$,
$S_1=(3,2)$,
$S_2=(6,3)$,
$S_3=(6,5)$,
$S_4=(7,5)$) 
that any path between $A$ and $S_j$ cannot have
more than $y_j-b$ NE-turns, of which one possible path is the one
which starts with a straight horizontal piece between $A$ and
$(b+x_j-y_j,b)$, and then continues with a zig-zag path until $S_j$
(beginning with an up-step and terminating with a right-step);
see again Figure~\ref{fig:2}, with $j=2$. Similarly, 
any path between $S_j$ and $B$ cannot have
more than $c-x_j$ NE-turns, of which one possible path is the one
which starts with a zig-zag path between $S_j$ and $(c,c-x_j+y_j)$
(beginning with an up-step and terminating with a right-step), 
and then continues with a straight vertical piece between $(c,c-x_j+y_j)$ and
$B$; Figure~\ref{fig:2} provides again an illustration. It is also
clear that all other $S_i$'s will lie weakly north-west of these path
portions. If we add 
the number of these NE-turns, 
then we obtain
\begin{align*} 
(y_j-b)+(c-x_j)=c-b-x_j+y_j.
\end{align*}
This agrees indeed with \eqref{eq:1}.
\end{proof}

We move on to the case of families of non-intersecting lattice paths.
The next lemma tells us 
the restriction that 
an upper ladder imposes on
the $i$-th path in a family of non-intersecting lattice paths, allowing
us to break the problem of finding the maximum total number of NE-turns
in families of non-intersecting lattice paths down to independent
maximisation problems for single paths (with the solution to the
latter problem being provided for by Lemma~\ref{lem:1}).

\begin{Lemma} \label{lem:2}
Let $A^{(i)}=(0,a_i)$, $E^{(i)}=(B-b_i,A)$, $i=1,2,\dots,n$, 
and $S_i=(x_i,y_i)$, $i=1,2,\dots,m$, be
lattice points in the plane with 
$a_1>a_2>\dots>a_n$ and $b_1>b_2>\dots>b_n$,
$0\le x_i\le B-b_n$ and $a_n\le y_i\le A$.
Then, in any family $(P_1,P_2,\dots,P_n)$ of non-intersecting lattice
paths, where $P_i$ runs from $A^{(i)}$ to $E^{(i)}$ and stays
weakly south-east of $S_k$, for $k=1,2,\dots,m$, 
the path $P_i$ has to 
stay weakly south-east of all points
\begin{multline} \label{eq:2}
\{(i-j,a_j-i+j): j=1,2,\dots,i\}
\cup
\{(B-b_j+i-j,A-i+j):j=1,2,\dots,i\}\\
\cup
\{S_k+(i-1,-i+1):k=1,2,\dots,m\}
\end{multline}
\end{Lemma}

\begin{proof}
Since the paths $P_1,P_2,\dots,P_n$ are non-intersecting, the paths
$P_{j+1},P_{j+2},\dots,P_{i-1}$ must stay between $P_j$ and $P_i$,
for all $j<i$. The point $A_j=(0,a_j)$ belongs to the path $P_j$,  
whereas the path 
$P_{j+1}$ must stay {\it strictly} south-east of
$P_j$. In particular, it must 
stay weakly south-east of $(1,a_j-1)$.
The same argument is repeated with $P_{j+1}$ and $P_{j+2}$, etc. The
claimed conclusion then follows without difficulty.
\end{proof}

%

\section{The main theorem for one-sided ladder regions}
\label{sec:main1}

We now apply the findings of the previous section to obtain our
first main result.

\begin{Theorem} \label{thm:1}
Let $A^{(i)}=(0,a_i)$, $E^{(i)}=(B-b_i,A)$, $i=1,2,\dots,n$, 
and $S_i=(x_i,y_i)$, $i=1,2,\dots,m$, be
lattice points in the plane with 
$a_1>a_2>\dots>a_n$ and $b_1>b_2>\dots>b_n$,
$0\le x_i\le B-b_n$ and $a_n\le y_i\le A$.
The maximum number of NE-turns which a family 
$(P_1,P_2,\dots,P_n)$ of non-intersecting lattice
paths, where $P_i$ runs from $A^{(i)}$ to $E^{(i)}$ and stays
weakly south-east of $S_k$, for $k=1,2,\dots,m$, can attain is
$$ \sum _{i=1} ^{n}t_i,
$$
where
\begin{multline*}
t_i=
B-a_i-b_i-
\max\big(\{-a_j+2(i-j),B-A-b_j+2(i-j):1\le j\le i\}\\
\cup
\{x_k-y_k+2(i-1):1\le k\le m\}\big).
\end{multline*}
\end{Theorem}

\begin{proof}
This follows immediately if Lemma~\ref{lem:2} is combined with
Lemma~\ref{lem:1} (with $a=0$, $b=a_i$, $c=B-b_i$, $d=A$,
and the points $S_i$
being the points in \eqref{eq:2}), by also taking Remark~\ref{rem:2}
into account.
\end{proof}

\begin{Example} \label{ex:1}
In order to illustrate Theorem~\ref{thm:1}, we choose 
$A=15$, $B=13$, $n=3$, $a_1=5$, $a_2=4$, $a_3=2$, $b_1=3$, $b_2=1$, $b_3=0$
(so that $A^{(1)}=(0,5)$, $A^{(2)}=(0,4)$, $A^{(3)}=(0,2)$, 
$B^{(1)}=(10,15)$, $B^{(2)}=(12,15)$, $B^{(3)}=(13,15)$),
$S_1=(4,6)$, $S_2=(8,9)$, and $S_3=(10,13)$.

\begin{figure}[h]
\tiny
$$
\Koordinatenachsen(14,16)(0,0)
\Kreis(0,5)
\Kreis(0,4)
\Kreis(0,2)
\Kreis(10,15)
\Kreis(12,15)
\Kreis(13,15)
\Pfad(0,5),11112121212121212222\endPfad
\Pfad(0,4),11111212121212121212222\endPfad
\Pfad(0,2),11111212121212121212122222\endPfad
\FeinPunkt(0,0)
\FeinPunkt(1,0)
\FeinPunkt(2,0)
\FeinPunkt(3,0)
\FeinPunkt(4,0)
\FeinPunkt(5,0)
\FeinPunkt(6,0)
\FeinPunkt(7,0)
\FeinPunkt(8,0)
\FeinPunkt(9,0)
\FeinPunkt(10,0)
\FeinPunkt(11,0)
\FeinPunkt(12,0)
\FeinPunkt(13,0)
\FeinPunkt(0,1)
\FeinPunkt(1,1)
\FeinPunkt(2,1)
\FeinPunkt(3,1)
\FeinPunkt(4,1)
\FeinPunkt(5,1)
\FeinPunkt(6,1)
\FeinPunkt(7,1)
\FeinPunkt(8,1)
\FeinPunkt(9,1)
\FeinPunkt(10,1)
\FeinPunkt(11,1)
\FeinPunkt(12,1)
\FeinPunkt(13,1)
\FeinPunkt(0,2)
\FeinPunkt(1,2)
\FeinPunkt(2,2)
\FeinPunkt(3,2)
\FeinPunkt(4,2)
\FeinPunkt(5,2)
\FeinPunkt(6,2)
\FeinPunkt(7,2)
\FeinPunkt(8,2)
\FeinPunkt(9,2)
\FeinPunkt(10,2)
\FeinPunkt(11,2)
\FeinPunkt(12,2)
\FeinPunkt(13,2)
\FeinPunkt(0,3)
\FeinPunkt(1,3)
\FeinPunkt(2,3)
\FeinPunkt(3,3)
\FeinPunkt(4,3)
\FeinPunkt(5,3)
\FeinPunkt(6,3)
\FeinPunkt(7,3)
\FeinPunkt(8,3)
\FeinPunkt(9,3)
\FeinPunkt(10,3)
\FeinPunkt(11,3)
\FeinPunkt(12,3)
\FeinPunkt(13,3)
\FeinPunkt(0,4)
\FeinPunkt(1,4)
\FeinPunkt(2,4)
\FeinPunkt(3,4)
\FeinPunkt(4,4)
\FeinPunkt(5,4)
\FeinPunkt(6,4)
\FeinPunkt(7,4)
\FeinPunkt(8,4)
\FeinPunkt(9,4)
\FeinPunkt(10,4)
\FeinPunkt(11,4)
\FeinPunkt(12,4)
\FeinPunkt(13,4)
\FeinPunkt(0,5)
\FeinPunkt(1,5)
\FeinPunkt(2,5)
\FeinPunkt(3,5)
\FeinPunkt(4,5)
\FeinPunkt(5,5)
\FeinPunkt(6,5)
\FeinPunkt(7,5)
\FeinPunkt(8,5)
\FeinPunkt(9,5)
\FeinPunkt(10,5)
\FeinPunkt(11,5)
\FeinPunkt(12,5)
\FeinPunkt(13,5)
\FeinPunkt(0,6)
\FeinPunkt(1,6)
\FeinPunkt(2,6)
\FeinPunkt(3,6)
\FeinPunkt(4,6)
\FeinPunkt(5,6)
\FeinPunkt(6,6)
\FeinPunkt(7,6)
\FeinPunkt(8,6)
\FeinPunkt(9,6)
\FeinPunkt(10,6)
\FeinPunkt(11,6)
\FeinPunkt(12,6)
\FeinPunkt(13,6)
\FeinPunkt(4,7)
\FeinPunkt(5,7)
\FeinPunkt(6,7)
\FeinPunkt(7,7)
\FeinPunkt(8,7)
\FeinPunkt(9,7)
\FeinPunkt(10,7)
\FeinPunkt(11,7)
\FeinPunkt(12,7)
\FeinPunkt(13,7)
\FeinPunkt(4,8)
\FeinPunkt(5,8)
\FeinPunkt(6,8)
\FeinPunkt(7,8)
\FeinPunkt(8,8)
\FeinPunkt(9,8)
\FeinPunkt(10,8)
\FeinPunkt(11,8)
\FeinPunkt(12,8)
\FeinPunkt(13,8)
\FeinPunkt(4,9)
\FeinPunkt(5,9)
\FeinPunkt(6,9)
\FeinPunkt(7,9)
\FeinPunkt(8,9)
\FeinPunkt(9,9)
\FeinPunkt(10,9)
\FeinPunkt(11,9)
\FeinPunkt(12,9)
\FeinPunkt(13,9)
%
\FeinPunkt(8,10)
\FeinPunkt(9,10)
\FeinPunkt(10,10)
\FeinPunkt(11,10)
\FeinPunkt(12,10)
\FeinPunkt(13,10)
%
\FeinPunkt(8,11)
\FeinPunkt(9,11)
\FeinPunkt(10,11)
\FeinPunkt(11,11)
\FeinPunkt(12,11)
\FeinPunkt(13,11)
%
\FeinPunkt(8,12)
\FeinPunkt(9,12)
\FeinPunkt(10,12)
\FeinPunkt(11,12)
\FeinPunkt(12,12)
\FeinPunkt(13,12)
\FeinPunkt(8,13)
\FeinPunkt(9,13)
\FeinPunkt(10,13)
\FeinPunkt(11,13)
\FeinPunkt(12,13)
\FeinPunkt(13,13)
%
\FeinPunkt(10,14)
\FeinPunkt(11,14)
\FeinPunkt(12,14)
\FeinPunkt(13,14)
%
\FeinPunkt(10,15)
\FeinPunkt(11,15)
\FeinPunkt(12,15)
\FeinPunkt(13,15)
\Label\lu{0}(0,0)
\Label\u{ 1}(1,0)
\Label\u{ 2}(2,0)
\Label\u{ 3}(3,0)
\Label\u{ 4}(4,0)
\Label\u{ 5}(5,0)
\Label\u{ 6}(6,0)
\Label\u{ 7}(7,0)
\Label\u{ 8}(8,0)
\Label\u{ 9}(9,0)
\Label\u{ 10}(10,0)
\Label\u{ 11}(11,0)
\Label\u{ 12}(12,0)
\Label\u{ 13}(13,0)
\Label\l{ 1}(0,1)
\Label\l{ 2}(0,2)
\Label\l{ 3}(0,3)
\Label\l{ 4}(0,4)
\Label\l{ 5}(0,5)
\Label\l{ 6}(0,6)
\Label\l{ 7}(0,7)
\Label\l{ 8}(0,8)
\Label\l{ 9}(0,9)
\Label\l{ 10}(0,10)
\Label\l{ 11}(0,11)
\Label\l{ 12}(0,12)
\Label\l{ 13}(0,13)
\Label\l{ 14}(0,14)
\Label\l{ 15}(0,15)
\normalsize
\Label\ro{\ \ A^{(1)}}(0,5)
\Label\ro{\ \ A^{(2)}}(0,4)
\Label\ro{\ \ A^{(3)}}(0,2)
\Label\u{E^{(1)}}(10,16)
\Label\u{E^{(2)}}(12,16)
\Label\u{E^{(3)}}(13,16)
\hskip7.7cm
$$
\caption{}
\label{fig:9}
\end{figure}

According to Theorem~\ref{thm:1}, we have
\begin{align*}
t_1&=13-5-3-
\max\big(\{-5,13-15-3\}\cup
\{4-6,8-9,10-13\}\big)=6,\\
t_2&=13-4-1-
\max\big(\{-5+2,13-15-3+2,-4,13-15-1\}\\
&\kern5cm
\cup
\{4-6+2,8-9+2,10-13+2\}\big)=7,\\
t_3&=13-2-0-
\max\big(\{-5+4,13-15-3+4,-4+2,13-15-1+2,\\
&\kern4cm
-2,13-15-0\}\cup
\{4-6+4,8-9+4,10-13+4\}\big)=8.
\end{align*}
Thus, the maximum number of NE-turns a family $(P_1,P_2,P_3)$ of
non-intersecting lattice paths, where $P_i$ runs from $A^{(i)}$ to
$E^{(i)}$, $i=1,2,3$, can have is $6+7+8=21$, with the individual
paths having at most $6$, $7$, $8$ NE-turns, respectively. 
An example of such a
family is shown in Figure~\ref{fig:9}.
\end{Example}

In view of Remark~\ref{rem:1}.(2), 
Corollary~\ref{cor:1} and Theorem~\ref{thm:1}, after little
simplification,
we obtain the following formula for the $a$-invariant
of one-sided ladder determinantal rings.

\begin{Corollary} \label{cor:2}
Let $Y=(Y_{i,j})_{0\le i\le A,\ 0\le j\le B}$ be
an upper ladder, and let $L$ be the associated ladder region.
Let $M=[u_1,u_2,\dots,u_n\mid v_1,v_2,\dots,v_n]$
be a bivector of positive integers with 
$u_1<u_2<\dots<u_n$ and $v_1<v_2<\dots<v_n$.
Then the $a$-invariant of
the ladder determinantal ring $R_M(Y)=K[Y]/I_M(Y)$ is given by
$$
\sum_{i=1}^n t_i-(A+B+1)n,
$$
where
\begin{multline*}
t_i=
\min\big(\{B+u_{n-j+1}-1-2(i-j),A+v_{n-j+1}-1-2(i-j):1\le j\le i\}\\
\cup
\{B-x_k+y_k-2(i-1):1\le k\le m\}\big),
\end{multline*}
where $S_k=(x_k,y_k)$, $k=1,2,\dots,m$, runs through 
the inwards corners of the upper ladder $L$.
\end{Corollary}

\begin{Example}
Let $A=15$, $B=13$, $n=3$, $L$ the ladder region indicated by the dots in
Figure~\ref{fig:9}, and $M=[3,5,6\mid 1,2,4]$.
(The reader should observe that $L$ is also the ladder region in
Figure~\ref{fig:5a}.)
The ladder can be ``described" by the inwards corners
$S_1=(4,6)$, $S_2=(8,9)$, and $S_3=(10,13)$.
(The reader should observe that the above choice of parameters leads to
the maximisation problem of Example~\ref{ex:1}.) 
We have
\begin{align*}
t_1&=
\min\big(\{13+6-1,15+4-1\}\\
&\kern4cm
\cup
\{13-4+6,13-8+9,13-10+13\}\big)=14,\\
t_2&=
\min\big(\{13+6-1-2,15+4-1-2,13+5-1,15+2-1-2\}\\
&\kern4cm
\cup
\{13-4+6-2,13-8+9-2,13-10+13-2\}\big)=12,\\
t_3&=
\min\big(\{13+6-1-4,15+4-1-4,13+5-1-2,15+2-1-2,\\
&\kern1cm
13+3-1,15+1-1\}
\cup
\{13-4+6-4,13-8+9-4,13-10+13-4\}\big)\\
&=10.
\end{align*}
Hence, the $a$-invariant
of the corresponding ladder determinantal ring $R_M(Y)$ equals
$$
(14+12+10)-(15+13+1)\cdot3=-51.
$$
\end{Example}

For the sake of comparison with Conca's formula \cite{ConcAD}
for the $a$-invariant of determinantal rings cogenerated by
a given minor (without ladder restriction), we provide the
specialisation of our result to that case separately.

\begin{Corollary} \label{cor:3}
Let $M=[u_1,u_2,\dots,u_n\mid v_1,v_2,\dots,v_n]$
be a bivector of positive integers with 
$u_1<u_2<\dots<u_n$ and $v_1<v_2<\dots<v_n$.
Then the $a$-invariant of
the determinantal ring $R_M(X)=K[X]/I_M(X)$ is given by
$$
\sum_{i=1}^n t_i-(A+B+1)n,
$$
where
$$
t_i=
\min\{B+u_{n-j+1}-1-2(i-j),A+v_{n-j+1}-1-2(i-j):1\le j\le i\}.
$$
\end{Corollary}

\begin{Example}
We let $B\le A$ and choose $M=[1,2,\dots,n\mid 1,2,\dots,n]$.
(For comparison, see \cite[Ex.~2.8]{ConcAB}.)
Then we obtain
\begin{align*}
t_i&=
\min\{B+n-2i+j,A+n-2i+j:1\le j\le i\}\\
&=B+n-2i+1.
\end{align*}
Hence, the $a$-invariant of $R_M(X)$ equals
$$
\sum_{i=1}^n(B+n-2i+1)-(A+B+1)n=Bn-(A+B+1)n=-(A+1)n,
$$
which is in accordance with \cite{GraeAA} and \cite[Cor.~1.5
with $f_i=0$ and $e_i=1$ for all~$i$]{BrHeAA}.
\end{Example}

\begin{Example}
We choose $M=[u_1,u_2,\dots,u_n\mid 1,2,\dots,n]$
with $u_i+1<u_{i+1}$ for $i=1,2,\dots,n-1$, and with $A-u_n\ge B-n$.
(For comparison, see \cite[Ex.~2.9]{ConcAB}.)
Then we obtain
\begin{align*}
t_i&=
\min\{B+u_{n-j+1}-1-2(i-j),A+n-2i+j:1\le j\le i\}\\
&=
\min\{B+u_{n-i+1}-1,A+n-2i+1\}.
\end{align*}
By our assumptions, we have 
$$A+n-2i+1\ge B+u_n-2i+1\ge B+u_{n-i+1}+2(i-1)-2i+1=B+u_{n-i+1}-1.$$
Hence, we have $t_i=B+u_{n-i+1}-1$, and
the $a$-invariant of $R_M(X)$ equals
\begin{align*}
\sum_{i=1}^n(B+u_{n-i+1}-1)-(A+B+1)n&=Bn+
\sum_{i=1}^nu_i-(A+B+2)n\\
&=\sum_{i=1}^nu_i-(A+2)n,
\end{align*}
which is in accordance with the result in \cite[Ex.~2.9]{ConcAB}.
\end{Example}

\section{How to achieve the maximum number of NE-turns: the two-sided
case}
\label{sec:aux2}

We now turn our attention to the two-sided case. 
We restrict our attention to the case where all $B^{(i)}$'s
lie completely in $L$, in order to avoid technical difficulties
resulting from the condition involving the $B^{(i)}$'s in
Theorem~\ref{thm:3}.

Again, 
we begin with the task of maximising the number of NE-turns of a
single path in a given ladder region, which is now two-sided.
The next lemma provides an algorithmic solution to the problem of
finding the maximum number of NE-turns of paths from a given
starting point to a given end point staying in a two-sided
ladder region. While, in view of Remark~\ref{rem:1}.(4), 
the set of lattice paths that we have to consider may actually be larger
(namely, it may include some paths which do not lie completely in the
ladder region), we will see later that it suffices to consider those
paths which do stay in the ladder.

The reader is advised to read the statement
below {\it in parallel\/} with the proof sketch that follows the
statement. Only then the motivation and meaning of the
individual steps of the algorithm will become apparent. While a
formal proof could be given, it would be unenlightening. This is
the reason we chose to provide a proof {\it sketch}, emphasising
the (geometric) ideas behind the construction.

\begin{Lemma} \label{lem:3}
Let $A=(a,b)$, $B=(c,d)$, and $S_i=(x_i,y_i)$ and $T_j=(z_j,w_j)$ be
lattice points with $a\le x_i,z_i\le c$ and $b\le y_i,w_i\le d$,
for $i=1,2,\dots,p$ and $j=1,2,\dots,q$. 
The maximum number of NE-turns which a lattice path from $A$ 
to $B$ which stays weakly south-east of $S_i$, for $i=1,2,\dots,p$, 
and weakly north-west of $T_j$, for $j=1,2,\dots,q$, the NE-turns
being different from any of the points $T_j$, $j=1,2,\dots,q$, can
attain can be
computed in the following manner:
\begin{enumerate}
\item \label{it1} 
Form the point set
$$
P_1=\{A,B\}\cup\{S_i:1\le i\le p\}
\cup\{T_j:1\le j\le q\}.
$$
\item \label{it2} 
Replace each point $(x,y)\in P_1$ by $(x+y,x-y)$.
Call the new point set $P_2$.
\item \label{it3} 
Order the points in $P_2$ according to the size of their first
coordinates, from smallest to largest. In the case of ties, order the
corresponding points arbitrarily. Let the result of this ordering be
$$
P_2=\{\hat A,U_1,U_2,\dots,U_{p+q},\hat B\}.
$$
Each $U_i$ is labelled $S$ or $T$, depending on whether it came from
a point $S_j$ or a point $T_j$, respectively. The last point, $\hat B$,
which came from $B$ is labelled by $S$ {\it and\/} $T$. 
\item \label{it4} 
Successively, form a new point set $P_3$. Initialise
$P_3=\{\hat A\}$. Scan through $U_1,U_2,\dots$ until a point labelled by $S$
is found with larger second coordinate than $A$, or until a point
labelled by $T$ is found with smaller second coordinate. If such a point
is found, add it to $P_3$.
If the added point
was $\hat B$, continue with \eqref{it6}, otherwise continue with
\eqref{it5}.
\item \label{it5}
If the last point added to $P_3$ was a point labelled with
$S$, say $C=U_i$,
then continue to scan through $U_{i+1},U_{i+2},\dots$, looking for a point
labelled by $S$ with larger second coordinate or for a point labelled
by $T$ with smaller second coordinate. If such a point is found,
then, in the first case, replace $C$ by 
this point, while, in
the second case, add 
the point found to $P_3$.
If the added point
was $\hat B$, continue with \eqref{it6}, otherwise continue with
\eqref{it5}. 

If the last point added to $P_3$ was a point labelled with
$T$, say $C=U_i$,
then continue to scan through $U_{i+1},U_{i+2},\dots$, looking for a point
labelled by $S$ with larger second coordinate or for a point labelled
by $T$ with smaller second coordinate. If such a point is found,
then, in the first case, add this 
point to $P_3$,
while, in the second case, replace $C$ by the 
point found.
If the added point
was $\hat B$, continue with \eqref{it6}, otherwise continue with
\eqref{it5}. 
\item \label{it6}
Let
$$
P_3=\{V^{(0)}=\hat A,V^{(1)},\dots,V^{(s)},V^{(s+1)}=\hat B\}.
$$
Compute the sum
\begin{equation} \label{eq:10} 
\frac {1} {2}\sum _{i=0} ^{s}
\min\left\{
V^{(i+1)}_1+V^{(i+1)}_2
-V^{(i)}_1-V^{(i)}_2,
V^{(i+1)}_1-V^{(i+1)}_2
-V^{(i)}_1+V^{(i)}_2\right\}.
\end{equation}
\end{enumerate}

The maximum is for instance realised by the path which connects
the points $S_i$ corresponding to the points in $P_3$ 
by zig-zag paths prepended by a horizontal
or vertical straight piece, as is necessary. More precisely,
given two successive points in $P_3$, the corresponding points
$S_i$ and $T_j$ {\em(}respectively $T_j$ and $S_i${\em)} 
are connected by a horizontal or
vertical straight piece {\em(}which may have length~$0${\em)} 
followed by a zig-zag path {\em(}which may also be empty{\em)} with a 
horizontal step at its end; cf.\ Figure~{\em\ref{fig:7}}.
\end{Lemma}

\begin{proof}[Sketch of proof]
While explaining what is behind the individual steps of the above
algorithm, we illustrate each of them by the running example in
which 
$A=(0,1)$,
$S_1=(2,2)$,
$S_2=(4,3)$,
$S_3=(2,5)$,
$S_4=(8,9)$,
$S_5=(10,10)$,
$S_6=(11,11)$,
$T_1=(4,1)$,
$T_2=(5,1)$,
$T_3=(6,1)$,
$T_4=(5,2)$,
$T_5=(5,5)$,
$T_6=(5,6)$,
$T_7=(8,7)$,
$T_8=(11,9)$,
$T_9=(13,10)$,
$B=(12,14)$; see Figure~\ref{fig:7}.
Clearly, there is nothing to be said about Step~(\ref{it1}) of the algorithm.

\begin{figure}[h]
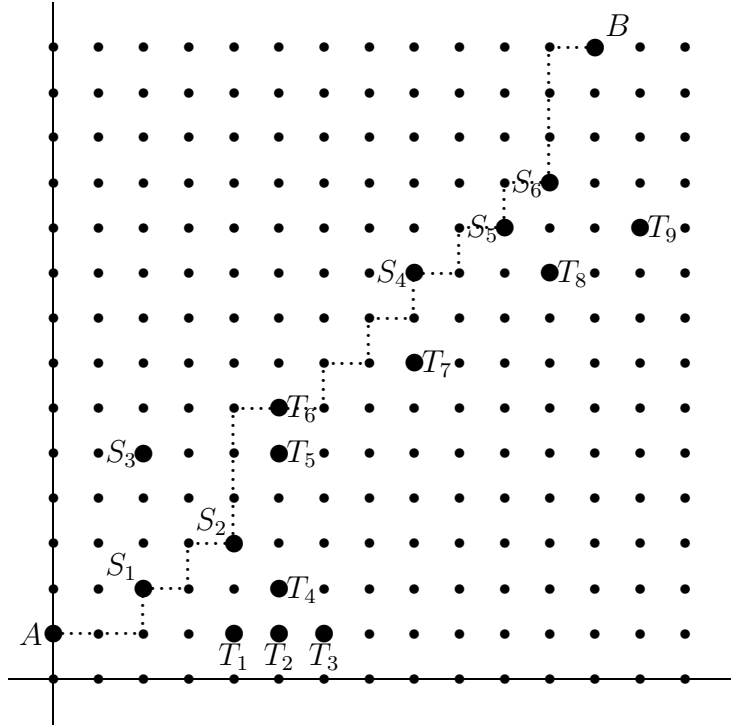

$$\Koordinatenachsen(15,15)(0,0)
\Gitter(15,15)(0,0)
\SPfad(0,1),1121212221121212121212221\endSPfad
\DickPunkt(0,1)
\DickPunkt(4,1)
\DickPunkt(5,1)
\DickPunkt(6,1)
\DickPunkt(2,2)
\DickPunkt(5,2)
\DickPunkt(4,3)
\DickPunkt(2,5)
\DickPunkt(5,5)
\DickPunkt(5,6)
\DickPunkt(8,9)
\DickPunkt(8,7)
\DickPunkt(10,10)
\DickPunkt(11,11)
\DickPunkt(11,9)
\DickPunkt(13,10)
\DickPunkt(12,14)
\Label\l{A}(0,1)
\Label\u{T_1}(4,1)
\Label\u{T_2}(5,1)
\Label\u{T_3}(6,1)
\Label\lo{S_1}(2,2)
\Label\r{T_4}(5,2)
\Label\lo{S_2}(4,3)
\Label\l{S_3}(2,5)
\Label\r{T_5}(5,5)
\Label\r{T_6}(5,6)
\Label\l{S_4}(8,9)
\Label\r{T_7}(8,7)
\Label\l{S_5}(10,10)
\Label\l{S_6}(11,11)
\Label\r{T_8}(11,9)
\Label\r{T_9}(13,10)
\Label\ro{B}(12,14)
\hskip8cm
$$
\caption{\small A lattice path with maximal number of NE-turns between
the $S_i$'s and the $T_i$'s}
\label{fig:7}
\end{figure}

From the arguments which proved Lemma~\ref{lem:1}, we know that
a lattice path which will attain the maximum number of NE-turns
should be as close as possible to a zig-zag path. So, the ``preferred"
(rough)
direction for our path is north-east, that is, the direction given by
the vector $(1,1)$. What may prevent us from going into that direction
from the very beginning until the very end is, first, the location
of starting and end point ($A$~and $B$ may not lie on a line parallel
to $(1,1)$), and, second, the points $S_i$, $i=1,2,\dots,p$, 
and $T_i$, $i=1,2,\dots,q$. If we imagine that we look
into the direction $(1,1)$ and imagine that we move into that direction,
then the situation that we encounter is as the one of a skier in a slalom:
there are some ``gates" which have to be passed on the right
(the points $S_i$) and some other gates which have to be passed
on the left (the points $T_i$). 

If we are on a particular line of the form $x+y=C$ (where $C$ is a fixed
integer), then along this line we may find some points $S_i$ and
some points $T_i$. Since the $S_i$'s have to be passed on the right
and the $T_i$'s on the left, along this fixed line it is clearly only
the right-most among the $S_i$'s and the left-most among the
$T_i$'s which are relevant. For example, if we consider the line
$x+y=7$ in our running example, then along this line we find
$S_2,S_3,T_3,T_4$, of which only $S_2$ (to be passed on the right) 
and $T_4$ (to be passed on the left) are really relevant, the other
points along this line can be disregarded.

In order to find a path from $A=(a,b)$ to $B=(c,d)$ which passes
weakly south-east (``to the right") of the points $S_i$ and
weakly north-west (``to the left") of the points $T_i$, we would
start at $A$ --- which is on the line $x+y=a+b$, and then proceed
to some point on the line $x+y=a+b+1$, then to some point on the
line $x+y=a+b+2$, \dots, and finally to $B$ --- which is on the line
$x+y=c+d$, and on each of these antidiagonal lines we will take
care that we pass to the right of the right-most point $S_i$ and
to the left of the left-most point $T_i$ on the line.

Given these observations, we can now understand what the meaning
of Step~(\ref{it2}) of the algorithm is. Upon replacement of a point
$(x,y)$ by $(x+y,x-y)$, the first coordinate, $x+y$, tells us
on which antidigonal line this point lies, and the second coordinate,
$x-y$, tells us how far right or left on that line the point lies.
The ordering of the points in $P_2$ performed in Step~(\ref{it3}) 
is then done so that first come
the points which are on the antidiagonal line $x+y=a+b$, then the
ones on $x+y=a+b+1$, \dots, and finally the ones on $x+y=c+d$.
This is exactly the order in which we have to consider these ``gates"
as we are advancing during our ``slalom run." In our running example,
we would obtain
\begin{multline*}
P_2=\{
(1,-1),
(4,0)_S,
(5,3)_T,
(6,4)_T,
(7,1)_S,
(7,-3)_S,
(7,5)_T,
(7,3)_T,
(10,0)_T,\\
(11,-1)_T,
(15,1)_T,
(17,-1)_S,
(20,2)_T,
(20,0)_S,
(22,0)_S,
(23,3)_T,
(26,-2)_{S,T}
\}.
\end{multline*}
(The labelling is indicated by subscripts.)

Steps~(\ref{it4}) and (\ref{it5}) take care that only 
``gates" are kept which are
relevant. The relevant ones are stored in the set $P_3$, 
while the rest of them is disregarded. In addition to the
above observation that along a line $x+y=C$ it is only the right-most
$S_i$ and the left-most $T_i$ which are relevant, there may be more
redundant points. Namely, in the proof of Lemma~\ref{lem:1} for 
the one-sided ladder case, we observed that only points $S_i=(x_i,y_i)$
with maximal $x_i-y_i$ are relevant, the other points can be ignored
(cf.\ \eqref{eq:1}). The same argument holds here. 
As a consequence, if the last point added to $P_3$ (=~the last ``relevant"
point) was a point corresponding to some $S_i$ (``a point labelled by
$S$"), then in Step~(\ref{it5}) we search for some 
$S_j$ which is more north-east
than $S_i$, and, if we find such an $S_j$, we may replace $S_i$ by $S_j$.
This is analogous for the points $T_i$. 

However, since we have a two-sided ladder region, while advancing we
must consider both sides. As in a real slalom, we may be forced to
``correct" our direction of movement if we encounter a $T$-point
which is more to the left than the last $S$-point, and vice versa.
This is also reflected in the instructions in Step~(\ref{it5}). 

For example, returning to our running example in Figure~\ref{fig:7},
when we start our ``slalom run" in $A$, then our first obstacle
which prevents us from doing a zig-zag path starting in $A$ is
the point $S_1$, which has to be passed on the right. Indeed, when
we apply Step~(\ref{it4}) of the algorithm to our example, then,
after initialisation $P_3=\{(1,-1)\}$, we encounter $(4,0)_S$ (which
corresponds to $S_1$) which
has larger second coordinate than $(1,-1)$ and is labelled by $S$,
and which consequently is added to $P_3$, so that we arrive at
$P_3=\{(1,-1),(4,0)\}$. The points $T_1$ and $T_2$ are no
obstacles, and, indeed, the corresponding points $(5,3)_T$ and $(6,4)_T$
in $P_2$, which are considered next during the execution of
Step~(\ref{it5}), do not have smaller second coordinate than $(4,0)$
and are therefore disregarded. Next comes $(7,1)_S$, corresponding to 
$S_2$. It has larger second coordinate than $(4,0)$ and, according to
Step~(\ref{it5}), replaces $(4,0)$ in $P_3$, so that we obtain
$P_3=\{(1,-1),(7,1)\}$. Indeed, once we include the restriction that
$S_2$ has to be passed on the right, the restriction imposed by $S_1$
becomes obsolete (see the dotted path; no other path between $A$
and $S_2$ can have more NE-turns). Continuing the execution of
Step~(\ref{it5}), the points $(7,-3)_S,
(7,5)_T,
(7,3)_T$ are all ignored. Then comes $(10,0)_T$ which has a smaller
second coordinate than $(7,1)$ and is labelled by $T$. According to
the algorithm, we have to add this point to $P_3$, so that we
obtain $P_3=\{(1,-1),(7,1),(10,0)\}$. Indeed, the corresponding point
$T_5$ prevents us from continuing a zig-zag path emanating from $S_2$,
we must ``correct" our move slightly to the left. Further continuation of
the application of Step~(\ref{it5}) will lead to a replacement of
$(10,0)$ by $(11,-1)$, the addition of $(20,0)$ and finally $(26,-2)$.
In other words, after application of Step~(\ref{it5}) we arrive at
$$
P_3=\{(1,-1),(7,1),(11,-1),(20,0),(26,-2)\}.
$$

After we have found the relevant ``gates" --- in form of a sequence
of points (which alternatingly correspond to points labelled by $S$ and $T$), by
Lemma~\ref{lem:1} with $n=0$, we must
now connect these points by as long as possible 
zig-zag pieces, supplemented by
some straight horizontal respectively vertical pieces as is
necessary to connect the zig-zag pieces. By the formula \eqref{eq:1}
with $n=0$, this leads directly to \eqref{eq:10}, once we recall
that the inverse of the mapping $(x,y)\mapsto(x+y,x-y)$ is given
by $(x,y)\mapsto\frac {1} {2}(x+y,x-y)$. In our running example
we obtain
$$
\frac {1} {2}
\left(
\min\left\{8,4\right\}
+\min\left\{2,6\right\}
+\min\left\{10,8\right\}
+\min\left\{4,8\right\}
\right)
=2+1+4+2=9.
$$
A path which achieves $9$~NE-turns is the dotted path in
Figure~\ref{fig:7}.
\end{proof}

The analogue of Lemma~\ref{lem:2} in the two-sided case is the
following.

\begin{Lemma} \label{lem:2b}
Let $A^{(i)}=(0,a_i)$, $E^{(i)}=(B-b_i,A)$, $i=1,2,\dots,n$, 
and $C_i=(x_i,y_i)$ and $D_j=(z_j,w_j)$ be
lattice points with 
$a_1>a_2>\dots>a_n$ and $b_1>b_2>\dots>b_n$,
$0\le x_i,z_i\le B-b_n$ and $a_n\le y_i,w_i\le A$,
for $i=1,2,\dots,p$ and $j=1,2,\dots,q$.
Then, in any family $(P_1,P_2,\dots,P_n)$ of non-intersecting lattice
paths, where $P_i$ runs from $A^{(i)}$ to $E^{(i)}$ and stays
weakly south-east of $C_k$ and weakly north-west of $D_j$, for
$k=1,2,\dots,p$ and $j=1,2,\dots,q$, 
the path $P_i$ has to
pass weakly south-east of all points
\begin{multline} \label{eq:2b}
\{(i-j,a_j-i+j): j=1,2,\dots,i-1\}
\cup
\{(B-b_j+i-j,A-i+j):j=1,2,\dots,i-1\}\\
\cup
\{C_k+(i-1,-i+1):k=1,2,\dots,p\}
\end{multline}
and weakly north-east of all points
\begin{equation} \label{eq:2c}
\{D_k+(-n+i,n-i):k=1,2,\dots,q\}.
\end{equation}
\end{Lemma}

\begin{proof}This is seen in the same way as in the proof of
Lemma~\ref{lem:2}.
\end{proof}

\section{The main theorem for two-sided ladder regions}
\label{sec:main2}

We now apply the findings of the previous section to obtain our
second main result.

\begin{Theorem}
\label{thm:2}
Let $A^{(i)}=(0,a_i)$, $E^{(i)}=(B-b_i,A)$, $i=1,2,\dots,n$, 
and $S_i=(x_i,y_i)$ and $T_j=(z_j,w_j)$ be
lattice points with 
$a_1>a_2>\dots>a_n$ and $b_1>b_2>\dots>b_n$,
$0\le x_i,z_i\le B-b_n$ and $a_n\le y_i,w_i\le A$,
for $i=1,2,\dots,p$ and $j=1,2,\dots,q$.
The maximum number of NE-turns which a family 
$(P_1,P_2,\dots,P_n)$ of non-intersecting lattice
paths, where $P_i$ runs from $A^{(i)}$ to $E^{(i)}$ and stays
weakly south-east of $S_k$ and weakly north-west of $T_j$, for
$k=1,2,\dots,p$ and $j=1,2,\dots,q$, can attain,
where a NE-turn of $P_i$ in a point $T_k+(-n+i,n-i)$ is not
counted, is
$$ \sum _{i=1} ^{n}t_i,
$$
where $t_i$ is the outcome of the algorithm described in
Lemma~{\em\ref{lem:3}} when applied to the special case where
$A=A^{(i)}$, $B=E^{(i)}$, the $S_i$'s being the points in
\eqref{eq:2b}, and the $T_i$'s being the points in \eqref{eq:2c}.
\end{Theorem}

In view of Remark~\ref{rem:1}.(2), 
Corollary~\ref{cor:1} and Theorem~\ref{thm:1}, 
the computation of the $a$-invariant
of two-sided ladder determinantal rings can be accomplished in the
following way. In the statement of the following corollary we 
need an extended meaning of the notion of inwards corners:
inwards corners along the upper boundary of a two-sided ladder
are defined in the same way as inwards corners of upper ladders,
while inwards corners along the lower boundary of a two-sided
ladder $L$ are points $(x,y)\in L$ for which both $(x+1,y)$ and $(x,y-1)$
are in $L$ but $(x+1,y-1)$ is not. For example, the inwards corners
along the lower boundary of the ladder region in Figure~\ref{fig:5} are 
$(6,8)$ and $(10,11)$.

\begin{Corollary} \label{cor:4}
Let $Y=(Y_{i,j})_{0\le i\le A,\ 0\le j\le B}$ be
a two-sided ladder, and let $L$ be the associated ladder region.
Let $M=[u_1,u_2,\dots,u_n\mid v_1,v_2,\dots,v_n]$
be a bivector of positive integers with 
$u_1<u_2<\dots<u_n$ and $v_1<v_2<\dots<v_n$.
Furthermore, let $A^{(i)}=(0,u_{n-i+1}-1)$ and $E^{(i)}=(B-v_{n-i+1}+1,A)$, 
$i=1,2,\dots,n$. 
We assume that all sets $B^{(i)}$ described in Theorem~\ref{thm:3} 
are completely contained in $L$, i.e., there exists at least one
family $\mathbf
P=(P_1,P_2,\dots,P_n)$ of non-intersecting lattice paths, $P_i$ running
from $A^{(i)}$ to $E^{(i)}$, which are completely contained in $L$.
Then the $a$-invariant of
the ladder determinantal ring $R_M(Y)=K[Y]/I_M(Y)$ is given by
$$
\sum_{i=1}^n (t_i+u_i+v_i)-(A+B+3)n,
$$
where $t_i$ is the outcome of the algorithm described in
Lemma~{\em\ref{lem:3}} when applied to the special case where
$A=A^{(i)}$, $B=E^{(i)}$, the $S_i$'s being the points in
\eqref{eq:2b} with $C_k$ running through all inwards corners
of the upper boundary of $L$, 
and the $T_i$'s being the points in \eqref{eq:2c} with $D_k$ running
through all inwards corners of the lower boundary of~$L$.
\end{Corollary}

\begin{figure}[h]
$$
\Pfad(-1,0),1111111122222211\endPfad
\Pfad(-1,2),112111222112111\endPfad
\Pfad(-1,4),11122221111111\endPfad
\DuennPunkt(0,0)
\DuennPunkt(1,0)
\DuennPunkt(2,0)
\DuennPunkt(0,1)
\DuennPunkt(1,1)
\DuennPunkt(2,1)
\DuennPunkt(3,1)
\DuennPunkt(0,2)
\DuennPunkt(1,2)
\DuennPunkt(2,2)
\DuennPunkt(3,2)
\DuennPunkt(0,3)
\DuennPunkt(1,3)
\DuennPunkt(2,3)
\DuennPunkt(3,3)
\DuennPunkt(0,4)
\DuennPunkt(1,4)
\DuennPunkt(2,4)
\DuennPunkt(3,4)
\DuennPunkt(4,4)
\DuennPunkt(5,4)
\DuennPunkt(6,4)
\DuennPunkt(0,5)
\DuennPunkt(1,5)
\DuennPunkt(2,5)
\DuennPunkt(3,5)
\DuennPunkt(4,5)
\DuennPunkt(5,5)
\DuennPunkt(6,5)
\DuennPunkt(7,5)
\DuennPunkt(8,5)
\DuennPunkt(0,6)
\DuennPunkt(1,6)
\DuennPunkt(2,6)
\DuennPunkt(3,6)
\DuennPunkt(4,6)
\DuennPunkt(5,6)
\DuennPunkt(6,6)
\DuennPunkt(7,6)
\DuennPunkt(8,6)
\DuennPunkt(0,7)
\DuennPunkt(1,7)
\DuennPunkt(2,7)
\DuennPunkt(3,7)
\DuennPunkt(4,7)
\DuennPunkt(5,7)
\DuennPunkt(6,7)
\DuennPunkt(7,7)
\DuennPunkt(8,7)
\DuennPunkt(0,8)
\DuennPunkt(1,8)
\DuennPunkt(2,8)
\DuennPunkt(3,8)
\DuennPunkt(4,8)
\DuennPunkt(5,8)
\DuennPunkt(6,8)
\DuennPunkt(7,8)
\DuennPunkt(8,8)
\DickPunkt(7,6)
\DickPunkt(1,3)
\DickPunkt(4,6)
\DickPunkt(6,7)
\DickPunkt(2,8)
\Label\o{P_n}(-1,0)
\Label\o{P_{n-1}}(-1,2)
\Label\o{P_{n-2}}(-1,4)
\Label\u{V}(2,0)
\Label\ru{W}(7,5)
\hskip7.5cm
\hbox{
\Pfad(-1,0),1112122211121211\endPfad
\Pfad(-1,2),112122112112111\endPfad
\Pfad(-1,4),11221122111111\endPfad
\DuennPunkt(0,0)
\DuennPunkt(1,0)
\DuennPunkt(2,0)
\DuennPunkt(0,1)
\DuennPunkt(1,1)
\DuennPunkt(2,1)
\DuennPunkt(3,1)
\DuennPunkt(0,2)
\DuennPunkt(1,2)
\DuennPunkt(2,2)
\DuennPunkt(3,2)
\DuennPunkt(0,3)
\DuennPunkt(1,3)
\DuennPunkt(2,3)
\DuennPunkt(3,3)
\DuennPunkt(0,4)
\DuennPunkt(1,4)
\DuennPunkt(2,4)
\DuennPunkt(3,4)
\DuennPunkt(4,4)
\DuennPunkt(5,4)
\DuennPunkt(6,4)
\DuennPunkt(0,5)
\DuennPunkt(1,5)
\DuennPunkt(2,5)
\DuennPunkt(3,5)
\DuennPunkt(4,5)
\DuennPunkt(5,5)
\DuennPunkt(6,5)
\DuennPunkt(7,5)
\DuennPunkt(8,5)
\DuennPunkt(0,6)
\DuennPunkt(1,6)
\DuennPunkt(2,6)
\DuennPunkt(3,6)
\DuennPunkt(4,6)
\DuennPunkt(5,6)
\DuennPunkt(6,6)
\DuennPunkt(7,6)
\DuennPunkt(8,6)
\DuennPunkt(0,7)
\DuennPunkt(1,7)
\DuennPunkt(2,7)
\DuennPunkt(3,7)
\DuennPunkt(4,7)
\DuennPunkt(5,7)
\DuennPunkt(6,7)
\DuennPunkt(7,7)
\DuennPunkt(8,7)
\DuennPunkt(0,8)
\DuennPunkt(1,8)
\DuennPunkt(2,8)
\DuennPunkt(3,8)
\DuennPunkt(4,8)
\DuennPunkt(5,8)
\DuennPunkt(6,8)
\DuennPunkt(7,8)
\DuennPunkt(8,8)
\DickPunkt(7,6)
\DickPunkt(1,3)
\DickPunkt(4,6)
\DickPunkt(6,7)
\DickPunkt(3,8)
\Label\o{P'_n}(-1,0)
\Label\o{P'_{n-1}}(-1,2)
\Label\o{P'_{n-2}}(-1,4)
\Label\u{V}(2,0)
\Label\ru{W}(7,5)
\hskip5cm
}
$$
\centerline{\small a. non-intersecting lattice paths
\hskip2cm
b. after push-up of paths}
\centerline{\small not entirely in $L$\hskip8cm}
\caption{Illustration of the ``push-up'' of lattice paths in
the proof of Corollary 17}
\label{fig:8}
\end{figure}

\begin{proof}
Given Theorem~\ref{thm:2}, this would be obvious, if there were not
the subtle difference between the conditions imposed on the
non-intersecting lattice paths in Theorem~\ref{thm:2} and the ones
in Theorem~\ref{thm:3}: in the latter theorem, lattice paths are 
allowed to leave the ladder region $L$ (cf.\ Remark~\ref{rem:1}.(4)),
while this is not the case if we apply Theorem~\ref{thm:2}
with the $C_k$'s and the $D_k$'s the inward corners of $L$. 
However, it is
easy to see that families of non-intersecting lattice paths, where
some of the paths protrude outside of $L$, all of its paths having
their NE-turns in $L^{(i)}\backslash B^{(i)}$, cannot achieve a higher
total number of NE-turns than families $(P_1,P_2,\dots,P_n)$ 
of paths which stay completely
inside $L$, even if one does not count the NE-turns of $P_i$
which lie in $B^{(i)}$. Indeed, let us consider a
family of non-intersecting lattice paths, where some its paths
run below the lower boundary of $L$; see the example on the left of
Figure~\ref{fig:8}. There, the ladder region $L$ is indicated by 
thin dots, and NE-turns of paths are indicated by bold dots.
Clearly, this means in particular that the
lowest path, $P_n$, has to run below $L$. Let $V$ be the
lattice point in $L$ which is the last point before $P_n$
leaves $L$, and let $W$ be the lattice point in $L$
where $P_n$ reenters $L$ after its ``excursion"; see
Figure~\ref{fig:8}. We replace the portion of $P_n$ lying
outside of $L$ by the path between $V$ and $W$ travelling along
the lower boundary of $L$. Since the paths in the family should
remain non-intersecting, we may have to ``push up" $P_{n-1}$,
$P_{n-2}$, etc., at the same time; see the right half of Figure~\ref{fig:8}.
These operations do not change the number of NE-turns of $P_i$ outside
the set $B^{(i)}$, $i=n,n-1,\dots$. 
Moreover, it should be observed that the points
$T_k+(-n+i,n-i)$ that would 
not be counted as NE-turns in
Theorem~\ref{thm:2} are points lying in $B^{(i)}$, so that this
corresponds well with the previous observation.
We do this ``push-up" for all portions of paths
which lie below $L$. In principle, these ``push-ups" may push
up $P_1$ beyond the upper boundary of~$L$. However, since
we assumed that all sets $B^{(i)}$ are completely contained in $L$,
this cannot happen. 

This completes the proof.
\end{proof}

\begin{Example}
We illustrate Corollary~\ref{cor:4} by choosing 
$A=15$, $B=13$, $n=3$, $L$ the ladder region indicated by the dots in
Figure~\ref{fig:10}, and $M=[3,5,6\mid 3,4,6]$.
(The reader should observe that $L$ is also the ladder region in
Figure~\ref{fig:5}.)
The ladder can be ``described" by the inwards corners
$C_1=(4,6)$, $C_2=(7,11)$, and $C_3=(8,12)$ along the upper boundary
of~$L$, and by the inwards corners
$D_1=(6,8)$ and $D_2=(10,11)$ along the lower boundary
of~$L$.

\begin{figure}[h]
\tiny
$$
\Koordinatenachsen(14,16)(0,0)
\FeinPunkt(0,0)
\FeinPunkt(1,0)
\FeinPunkt(2,0)
\FeinPunkt(3,0)
\FeinPunkt(4,0)
\FeinPunkt(5,0)
\FeinPunkt(6,0)
%
\FeinPunkt(0,1)
\FeinPunkt(1,1)
\FeinPunkt(2,1)
\FeinPunkt(3,1)
\FeinPunkt(4,1)
\FeinPunkt(5,1)
\FeinPunkt(6,1)
%
\FeinPunkt(0,2)
\FeinPunkt(1,2)
\FeinPunkt(2,2)
\FeinPunkt(3,2)
\FeinPunkt(4,2)
\FeinPunkt(5,2)
\FeinPunkt(6,2)
%
\FeinPunkt(0,3)
\FeinPunkt(1,3)
\FeinPunkt(2,3)
\FeinPunkt(3,3)
\FeinPunkt(4,3)
\FeinPunkt(5,3)
\FeinPunkt(6,3)
%
\FeinPunkt(0,4)
\FeinPunkt(1,4)
\FeinPunkt(2,4)
\FeinPunkt(3,4)
\FeinPunkt(4,4)
\FeinPunkt(5,4)
\FeinPunkt(6,4)
%
\FeinPunkt(0,5)
\FeinPunkt(1,5)
\FeinPunkt(2,5)
\FeinPunkt(3,5)
\FeinPunkt(4,5)
\FeinPunkt(5,5)
\FeinPunkt(6,5)
%
\FeinPunkt(0,6)
\FeinPunkt(1,6)
\FeinPunkt(2,6)
\FeinPunkt(3,6)
\FeinPunkt(4,6)
\FeinPunkt(5,6)
\FeinPunkt(6,6)
%
\FeinPunkt(4,7)
\FeinPunkt(5,7)
\FeinPunkt(6,7)
%
\FeinPunkt(4,8)
\FeinPunkt(5,8)
\FeinPunkt(6,8)
\FeinPunkt(7,8)
\FeinPunkt(8,8)
\FeinPunkt(9,8)
\FeinPunkt(10,8)
%
\FeinPunkt(4,9)
\FeinPunkt(5,9)
\FeinPunkt(6,9)
\FeinPunkt(7,9)
\FeinPunkt(8,9)
\FeinPunkt(9,9)
\FeinPunkt(10,9)
%
\FeinPunkt(4,10)
\FeinPunkt(5,10)
\FeinPunkt(6,10)
\FeinPunkt(7,10)
\FeinPunkt(8,10)
\FeinPunkt(9,10)
\FeinPunkt(10,10)
%
\FeinPunkt(4,11)
\FeinPunkt(5,11)
\FeinPunkt(6,11)
\FeinPunkt(7,11)
\FeinPunkt(8,11)
\FeinPunkt(9,11)
\FeinPunkt(10,11)
\FeinPunkt(11,11)
\FeinPunkt(12,11)
\FeinPunkt(13,11)
%
\FeinPunkt(7,12)
\FeinPunkt(8,12)
\FeinPunkt(9,12)
\FeinPunkt(10,12)
\FeinPunkt(11,12)
\FeinPunkt(12,12)
\FeinPunkt(13,12)
\FeinPunkt(8,13)
\FeinPunkt(9,13)
\FeinPunkt(10,13)
\FeinPunkt(11,13)
\FeinPunkt(12,13)
\FeinPunkt(13,13)
\FeinPunkt(8,14)
\FeinPunkt(9,14)
\FeinPunkt(10,14)
\FeinPunkt(11,14)
\FeinPunkt(12,14)
\FeinPunkt(13,14)
\FeinPunkt(8,15)
\FeinPunkt(9,15)
\FeinPunkt(10,15)
\FeinPunkt(11,15)
\FeinPunkt(12,15)
\FeinPunkt(13,15)
\Kreis(0,5)
\Kreis(0,4)
\Kreis(0,2)
\Kreis(8,15)
\Kreis(10,15)
\Kreis(11,15)
\Label\lu{0}(0,0)
\Label\u{ 1}(1,0)
\Label\u{ 2}(2,0)
\Label\u{ 3}(3,0)
\Label\u{ 4}(4,0)
\Label\u{ 5}(5,0)
\Label\u{ 6}(6,0)
\Label\u{ 7}(7,0)
\Label\u{ 8}(8,0)
\Label\u{ 9}(9,0)
\Label\u{ 10}(10,0)
\Label\u{ 11}(11,0)
\Label\u{ 12}(12,0)
\Label\u{ 13}(13,0)
\Label\l{ 1}(0,1)
\Label\l{ 2}(0,2)
\Label\l{ 3}(0,3)
\Label\l{ 4}(0,4)
\Label\l{ 5}(0,5)
\Label\l{ 6}(0,6)
\Label\l{ 7}(0,7)
\Label\l{ 8}(0,8)
\Label\l{ 9}(0,9)
\Label\l{ 10}(0,10)
\Label\l{ 11}(0,11)
\Label\l{ 12}(0,12)
\Label\l{ 13}(0,13)
\Label\l{ 14}(0,14)
\Label\l{ 15}(0,15)
\normalsize
\Label\ro{\ \ A^{(1)}}(0,5)
\Label\ro{\ \ A^{(2)}}(0,4)
\Label\ro{\ \ A^{(3)}}(0,2)
\Label\u{E^{(1)}}(8,16)
\Label\u{E^{(2)}}(10,16)
\Label\u{E^{(3)}}(11,16)
\Pfad(0,5),111212222112121222\endPfad
\Pfad(0,4),111121222211212122221\endPfad
\Pfad(0,2),111121212222112121222212\endPfad
\DickPunkt(3,6)
\DickPunkt(6,11)
\DickPunkt(7,12)
\DickPunkt(4,5)
\DickPunkt(7,10)
\DickPunkt(8,11)
\DickPunkt(9,15)
\DickPunkt(4,3)
\DickPunkt(5,4)
\DickPunkt(8,9)
\DickPunkt(9,10)
\DickPunkt(10,14)
\hskip7.7cm
$$
\caption{}
\label{fig:10}
\end{figure}

We now have to compute the quantities~$t_i$, $i=1,2,3$. 
First, we have to apply the algorithm of Lemma~\ref{lem:3} with
$A=A^{(1)}=(0,5)$, $E=E^{(1)}=(8,15)$, 
$S_1=(4,6)$, $S_2=(7,11)$, $S_3=(8,12)$,
$T_1=(4,10)$, and $T_2=(8,13)$. The point set in Step~(1) is
$$
P_1=\{(0,5),\,(8,15),\,(4,6),\,(7,11),\,(8,12),\,
(4,10),\,(8,13)\}.
$$
Next, the point set in Step~(2) is
$$
P_2=\{(5,-5),\,(23,-7),\,(10,-2),\,(18,-4),\,(20,-4),\,
(14,-6),\,(21,-5)\},
$$
while, after the ordering and labelling in Step~(3), it is
$$
P_2=\{(5,-5),\,(10,-2)_S,\,(14,-6)_T,\,(18,-4)_S,\,
(20,-4)_S,\,(21,-5)_T,\,(23,-7)_{S,T}\}.
$$
The point set obtained in Step~(4) is
$$
P_3=\{(5,-5),\,(10,-2)_S,\,(14,-6)_T,\,(18,-4)_S,\,
(23,-7)_{S,T}\}.
$$
Hence, we have
$$
t_1=\frac {1} {2}\big(
\min\{8,2\}+
\min\{0,8\}+
\min\{6,2\}+
\min\{2,8\}
\big)=3.
$$
The first path in Figure~\ref{fig:10} is a path with that number of
(valid) NE-turns. In the figure, the (valid) NE-turns are indicated
as thick dots. The point $(4,10)$ is not a valid NE-turn since it
lies in the set $B^{(1)}$ (cf.\ the statement of Theorem~\ref{thm:3}).

In order to perform the same computation for obtaining~$t_2$, we 
have to apply the algorithm of Lemma~\ref{lem:3} with
$A=A^{(2)}=(0,4)$, $E=E^{(1)}=(10,15)$, 
$S_1=(5,5)$, $S_2=(8,10)$, $S_3=(9,11)$,
$T_1=(5,9)$, and $T_2=(9,12)$. We get
$$
P_3=\{
(4,-4),\,(10,0)_S,\,(14,-4)_T,\,(18,-2)_S,\,(25,-5)_{S,T}
\},
$$
so that
$$
t_2=\frac {1} {2}\big(
\min\{10,2\}+
\min\{0,8\}+
\min\{6,2\}+
\min\{4,10\}
\big)=4.
$$
The second path in Figure~\ref{fig:10} is a path with that number of
(valid) NE-turns.

Finally, in order to perform the computation for obtaining~$t_3$,
we get 
$$
P_3=\{
(2,-2),\,(10,2)_S,\,
(14,-2)_T,\,(18,0)_S,\,(26,-4)_{S,T}
\},
$$
so that
$$
t_3=\frac {1} {2}\big(
\min\{12,4\}+
\min\{0,8\}+
\min\{6,2\}+
\min\{4,12\}
\big)=5.
$$
The third path in Figure~\ref{fig:10} is a path with that number of
(valid) NE-turns.

Consequently, the $a$-invariant of $R_M(Y)$ equals
$$
(3+4+5+3+5+6+3+4+6)-(15+13+3)\cdot 3=-54.
$$
\end{Example}


\begin{thebibliography}{10}

\bibitem{AbhyAB} 
S. S. Abhyankar, {\em Enumerative
combinatorics of Young tableaux}, Marcel Dekker, New York,
Basel, 1988.  

\bibitem{AbKuAC}
S. S. Abhyankar and D. M. Kulkarni, {\em 
On Hilbertian ideals}, Linear Alg. Appl. {\bf 116} (1989), 53--76.

\bibitem{BrHeAA}
W.    Bruns and J. Herzog, {\em On the
computation of $a$-invariants}, manuscripta math. {\bf 77} (1992),
201--213.  


\bibitem{ConcAB}
A.    Conca, {\em Ladder determinantal rings}, 
J.~Pure Appl. Algebra  {\bf 98} (1995), 119--134.

\bibitem{ConcAD}
A.    Conca, {\em The $a$-invariant of determinantal ideals},
Math. J. Toyama Univ. {\bf 18} (1995), 47--63.


\bibitem{CoHeAB}
A.    Conca and J. Herzog, {\em 
Ladder determinantal rings have rational
 singularities}, Adv. Math. {\bf132} (1997), 120--147. 

\bibitem{GhorAC}
S. R. Ghorpade, {\em Abhyankar's work on 
Young tableaux and some recent developments}, in: Proc. Conf. on 
Algebraic Geometry and Its Applications (Purdue Univ., June 1990), 
Sprin\-ger--Ver\-lag, New York, 1994, pp.~215--249.

\bibitem{GhorAD}
S. R. Ghorpade, {\em Young bitableaux, 
lattice paths and Hilbert functions}, J. Statist. Plann. Inference
{\bf 54} (1996), 55--66.

\bibitem{GhorAF}
S. R. Ghorpade, {\em Hilbert functions of ladder determinantal 
varieties}, Discrete Math. {\bf 246} (2002), 131--175.

\bibitem{GoWaAA}
S. Goto and K. Watanabe, {\em On graded rings I}, J. Math. Soc.
Japan {\bf 30} (1978), 179--213.

\bibitem{GraeAA}
H. G. Gr\"abe, {\em Streckungsringe}, 
Dissertation~B, P\"adagogische Hochschule
``Dr. Theodor Neubauer", Erfurt, DDR, 1988. 

\bibitem{HeTrAA}
J.    Herzog and N. V. Trung, {\em
Gr\"obner bases and multiplicity of determinantal and Pfaffian
ideals}, Adv. Math. {\bf 96} (1992), 1--37.  

\bibitem{KnMiAA}
A.    Knutson and E. Miller, {\em Gr\"obner geometry of Schubert 
polynomials}, Ann. Math.~(2) {\bf 161} (2005), 1245--1318.

\bibitem{KrPrAA}
C.    Krattenthaler and M. Prohaska, {\em A remarkable formula for 
counting nonintersecting lattice paths in a ladder with respect to
turns}, Trans. Amer. Math. Soc. {\bf 351} (1999), 1015--1042.

\bibitem{KrRuAA}
C.    Krattenthaler and M. Rubey, {\em
A determinantal formula for the Hilbert series of 
one-sided ladder determinantal rings}, in: Algebra, Arithmetic 
and Geometry with Applications, (C.~Christensen, G.~Sundaram, 
A.~Sathaye and C.~Bajaj, eds.),
Springer--Verlag, New York, (2004), pp.~337--356.

\bibitem{KulkAD}
D. M. Kulkarni, {\em Hilbert polynomial
of a certain ladder-determinantal ideal}, J. Algebraic Combin.
{\bf 2}, (1993), 57--72. 



\bibitem{NaraAA}
H. Narasimhan, {\em The irreducibility of ladder determinantal
varieties}, J. Algebra {\bf 102} (1986), 162--185.

\bibitem{RubeAC} M. Rubey, 
{\it The $h$-vector of a ladder determinantal ring cogenerated by
  $2\times 2$ minors is log-concave},
J. Algebra {\bf 292} (2005), 303--323.

\bibitem{StanDA}
R. P. Stanley, {\em Hilbert functions of graded algebras},
Adv. Math. {\bf 28} (1978), 57--83.

\bibitem{Wang}
H.-J. Wang, {\em  A determinantal formula for the Hilbert series of determinantal rings of one-sided ladder},  J. Algebra {\bf 265} (2003),  79–-99.

\end{thebibliography}
\end{document}

Here is my proposed answer for the maximum number of turns for given
starting and end points. I assume the starting points to be $(a_1,0)$,
$(a_2,0)$, \dots, $(a_n,0)$, and the end points to be $(0,b_1)$,
$(0,b_2)$, \dots, $(0,b_n)$, $a_1<a_2<\dots<a_n$, $b_1<b_2<\dots<b_n$.
Let ${\bf a}=(a_1,a_2,\dots,a_n)$ and ${\bf b}=(b_1,b_2,\dots,b_n)$.
I believe the maximum number of NE-turns of a family of non-intersecting
lattice paths connecting these starting and end points to be
$$ \sum _{i=1} ^{n}t_i
$$
where
$$t_i:=\left\{\begin{array}{ll}b_i- (a_1-a_i+2(i-1))_+
-(b_1-a_i+2(i-1)-(a_1-a_i+2(i-1))_+)_+&a_i\ge b_i\\ 
a_i- (b_1-b_i+2(i-1))_+
-(a_1-b_i+2(i-1)-(b_1-b_i+2(i-1))_+)_+&a_i< b_i.\\ 
\end{array}\right.
$$
By $(\alpha)_+$ I mean $\alpha$ if $\alpha\ge0$ and $0$ otherwise. 
Here is an example with ${\bf a}=(3,4,5,7)$ and ${\bf b}=(2,3,8,9)$ (also
using matrix notation for the coordinates). You
have to rotate everything by $180^\circ$ to get your picture.

\begin{figure}[h]
$$
\Gitter(10,8)(0,0)
\Pfad(0,0),2111212121212112\endPfad
\Pfad(0,2),1121212121211\endPfad
\Pfad(0,3),1221212\endPfad
\Pfad(0,4),22121\endPfad
\DickPunkt(0,0)
\DickPunkt(0,2)
\DickPunkt(0,3)
\DickPunkt(0,4)
\DickPunkt(2,7)
\DickPunkt(3,7)
\DickPunkt(8,7)
\DickPunkt(9,7)
\Kreis(0,1)
\Kreis(3,2)
\Kreis(4,3)
\Kreis(5,4)
\Kreis(6,5)
\Kreis(7,6)
\Kreis(2,3)
\Kreis(3,4)
\Kreis(4,5)
\Kreis(5,6)
\Kreis(6,7)
\Kreis(1,5)
\Kreis(2,6)
\Kreis(0,6)
\Kreis(1,7)
\PfadDicke{.4pt}
\Pfad(0,0),222222211111111\endPfad
\hskip6cm
$$
\caption{}
\label{fig:3}
\end{figure}

\begin{align*}
t_1&=\min\!\big\{3,2,3+2-0\big\}=2,\\
t_2&=\min\!\big\{4,3,4+3-\max\{3+2,2+2\}\big\}=2,\\
t_3&=\min\!\big\{5,8,5+8-
\max\{3+4,2+4,4+2,3+2\}\big\}=5,\\
t_4&=\min\!\big\{7,9,7+9-
\max\{3+6,2+6,
4+4,3+4,
5+2,8+2\}\big\}=6.
\end{align*}

$$\Gitter(6,5)(0,0)
\DickPunkt(1,4)
\DickPunkt(2,4)
\DickPunkt(3,4)
\DickPunkt(5,4)
\Pfad(0,3),21\endPfad
\Pfad(1,2),212\endPfad
\Pfad(2,1),2122\endPfad
\Pfad(3,0),212221\endPfad
\Kreis(0,4)
\Kreis(1,3)
\Kreis(2,2)
\Kreis(3,1)
\Kreis(4,4)
\hskip3cm
$$

